\documentclass[11pt]{article}

\usepackage{amsmath, amsfonts, amssymb, amsthm,  graphicx, enumerate}
\usepackage[letterpaper, left=1truein, right=1truein, top = 1truein, bottom = 1truein]{geometry}

\usepackage[numbers, square]{natbib}

\usepackage[
            CJKbookmarks=true,
            bookmarksnumbered=true,
            bookmarksopen=true,
            colorlinks=true,
            citecolor=red,
            linkcolor=blue,
            anchorcolor=red,
            urlcolor=blue,
            pdfauthor={Zongming Ma, Yihong Wu}
            ]{hyperref}

\usepackage{hyperref}
\usepackage[usenames]{color}
\newcommand{\red}{\color{red}}
\newcommand{\blue}{\color{blue}}

\usepackage{prettyref,soul,xspace}

\newcommand{\eg}{e.g.\xspace}
\newcommand{\ie}{i.e.\xspace}
\newcommand{\iid}{i.i.d.\xspace}

\newcommand{\ones}{\mathbf 1}
\newcommand{\reals}{{\mathbb{R}}}
\newcommand{\integers}{{\mathbb{Z}}}
\newcommand{\naturals}{{\mathbb{N}}}

\newcommand{\supp}{{\rm supp}}
\newcommand{\eexp}{{\rm e}}

\newcommand{\diff}{{\rm d}}

\newcommand{\rank}{\mathop{\sf rank}}
\newcommand{\Tr}{\mathop{\sf Tr}}
\newcommand{\diag}{\mathop{\text{diag}}}

\newcommand{\Expect}{\mathbb{E}}
\newcommand{\expect}[1]{\mathbb{E} #1 }

\newcommand{\Prob}{\mathbb{P}}
\newcommand{\prob}[1]{{ \mathbb{P}\left\{ #1 \right\} }}

\newcommand{\eqdistr}{{\stackrel{\rm (d)}{=}}}

\newcommand{\Th}{{^{\rm th}}}

\newrefformat{eq}{(\ref{#1})}
\newrefformat{chap}{Chapter~\ref{#1}}
\newrefformat{sec}{Section~\ref{#1}}
\newrefformat{algo}{Algorithm~\ref{#1}}
\newrefformat{fig}{Fig.~\ref{#1}}
\newrefformat{tab}{Table~\ref{#1}}
\newrefformat{rmk}{Remark~\ref{#1}}
\newrefformat{clm}{Claim~\ref{#1}}
\newrefformat{def}{Definition~\ref{#1}}
\newrefformat{cor}{Corollary~\ref{#1}}
\newrefformat{lmm}{Lemma~\ref{#1}}
\newrefformat{lemma}{Lemma~\ref{#1}}
\newrefformat{prop}{Proposition~\ref{#1}}
\newrefformat{app}{Appendix~\ref{#1}}
\newrefformat{ex}{Example~\ref{#1}}
\newrefformat{exer}{Exercise~\ref{#1}}
\newrefformat{soln}{Solution~\ref{#1}}

\newcommand{\lunder}[1]{{\underset{\raise0.3em\hbox{$\smash{\scriptscriptstyle-}$}}{#1}}}

\newcommand{\floor}[1]{{\left\lfloor {#1} \right \rfloor}}
\newcommand{\ceil}[1]{{\left\lceil {#1} \right \rceil}}

\newcommand{\norm}[1]{\left\|{#1} \right\|}

\newcommand{\lnorm}[2]{\left\|{#1} \right\|_{{#2}}}

\newcommand{\Fnorm}[1]{\lnorm{#1}{\rm F}}
\newcommand{\fnorm}[1]{\|#1\|_{\rm F}}

\newcommand{\opnorm}[1]{\|#1\|_{\rm op}}

\newcommand{\iprod}[2]{\left \langle #1, #2 \right\rangle}
\newcommand{\Iprod}[2]{\langle #1, #2 \rangle}

\newcommand{\indc}[1]{{\mathbf{1}_{\left\{{#1}\right\}}}}

\def\innergetnumber#1[#2]#3{#2}
\def\getnumber{\expandafter\innergetnumber\jobname}

\newcommand{\bszero}{{\boldsymbol{0}}}

\newcommand{\bbB}{{\mathbb{B}}}

\newcommand{\bbS}{{\mathbb{S}}}

\newcommand{\sfK}{{\mathsf{K}}}

\newcommand{\sfS}{{\mathsf{S}}}

\newcommand{\calF}{{\mathcal{F}}}

\newcommand{\calM}{{\mathcal{M}}}
\newcommand{\calN}{{\mathcal{N}}}

\newcommand{\calT}{{\mathcal{T}}}

\newcommand{\calX}{{\mathcal{X}}}

\newcommand{\bfA}{{\mathbf{A}}}
\newcommand{\bfB}{{\mathbf{B}}}

\newcommand{\bfI}{{\mathbf{I}}}

\newcommand{\bfX}{{\mathbf{X}}}
\newcommand{\bfY}{{\mathbf{Y}}}
\newcommand{\bfZ}{{\mathbf{Z}}}

\newcommand{\tG}{{\tilde{G}}}

\newcommand{\tI}{{\tilde{I}}}

\newcommand{\tM}{{\tilde{M}}}

\newcommand{\tS}{{\tilde{S}}}

\newcommand{\tW}{{\tilde{W}}}

\newcommand{\comp}[1]{{#1^{\rm c}}}

\newcommand{\ntok}[2]{{#1,\ldots,#2}}

\newcommand{\pth}[1]{\left( #1 \right)}

\newcommand{\sth}[1]{\left\{ #1 \right\}}

\newcommand{\KL}[2]{D(#1 \, || \, #2)}

\newcommand{\rsupp}{{\rm supp}_{\rm r}}
\newcommand{\csupp}{{\rm supp}_{\rm c}}
\newcommand{\hI}{\widehat{I}}
\newcommand{\hJ}{\widehat{J}}

\newcommand{\RR}{\mathbb{R}}

\newcommand{\tr}{\mathsf{Tr}}

\newcommand{\wh}{\widehat}
\newcommand{\wt}{\widetilde}

\newcommand{\wpal}{\rm with probability at least~}

\newcommand{\kfnorm}[2]{\|{#1}\|_{(#2)}}
\newcommand{\dKL}{d_{\rm KL}}

\renewcommand{\bfI}{I}
\renewcommand{\bfA}{A}
\renewcommand{\bfB}{B}
\renewcommand{\bfX}{X}
\renewcommand{\bfY}{Y}
\renewcommand{\bfZ}{Z}

\renewcommand{\AA}{{A}}

\newcommand{\TT}{\Theta}

\newcommand{\A}{{A}}

\newcommand{\K}{{K}}

\newcommand{\X}{{X}}
\newcommand{\Z}{{Z}}

\newcommand{\x}{{x}}

\newcommand{\LL}{\Lambda}
\renewcommand{\SS}{\Sigma}

\newcommand{\bfSigma}{\Sigma}

\def\tSS{\widetilde{\bfSigma}}

\def\hLL{\widehat{\LL}}

\newcommand{\sqnorm}[1]{\|#1\|_{{\rm S}_q}}
\newcommand{\lsqnorm}[1]{\left\|#1\right\|_{{\rm S}_q}}
\newcommand{\lsqnormdual}[1]{\left\|#1\right\|_{{\rm S}_{q^*}}}
\newcommand{\Lsqnorm}[1]{\|#1\|_{{\rm S}_q}}
\newcommand{\snorm}[2]{\|#2\|_{{\rm S}_{#1}}}

\newcommand{\sym}{\mathrm{S}}

\newcommand{\Poisson}{\mathrm{Poisson}}

\newcommand{\col}[2]{{#1}_{* {#2}}}
\newcommand{\row}[2]{{#1}_{{#2} *}}

\newcommand{\nb}[1]{{\sf\blue[#1]}}
\newcommand{\nbr}[1]{{\sf\red[#1]}}

\theoremstyle{plain}
\newtheorem{theorem}{Theorem}
\newtheorem*{theorem*}{Theorem}

\newtheorem{proposition}{Proposition}
\newtheorem{lemma}{Lemma}
\newtheorem{claim}{Claim}

\theoremstyle{definition}

\newtheorem{remark}{Remark}
\newtheorem{example}{Example}

\newcommand{\beq}{\begin{equation}}
\newcommand{\eeq}{\end{equation}}

\newcommand{\Hyper}{{\rm Hypergeometric}}
\newcommand{\vol}{{\rm vol}}

\newcommand{\ui}{unitarily invariant\xspace}

\title{Volume Ratio, Sparsity, and Minimaxity under\\ 
Unitarily Invariant Norms}
\author{
Zongming Ma\thanks{Z. Ma is with the Department of Statistics, The Wharton School, University of Pennsylvania, Philadelphia, PA 19104, USA.
Email: {\ttfamily zongming@wharton.upenn.edu}.}
~~~and~~~Yihong Wu\thanks{
Y. Wu is with Department of Electrical and Computer Engineering, University of Illinois Urbana-Champaign, Urbana, IL 61801, USA. Email: {\ttfamily yihongwu@illinois.edu}.}
}
\date{~}
\begin{document}
\maketitle
%

\begin{abstract}
	The current paper presents a novel machinery for studying non-asymptotic minimax estimation of high-dimensional matrices, which yields tight minimax rates for a large collection of loss functions in a variety of problems. 
	
	Based on the convex geometry of finite-dimensional Banach spaces, we first develop a volume ratio approach for determining minimax estimation rates of unconstrained normal mean matrices under all squared unitarily invariant norm losses. In addition, we establish the minimax rates for estimating mean matrices with submatrix sparsity, where the sparsity constraint introduces an additional term in the rate whose dependence on the norm differs completely from the rate of the unconstrained problem. Moreover, the approach is applicable to the matrix completion problem under the low-rank constraint.
	
The new method also extends beyond the normal mean model. In particular, it yields tight rates in covariance matrix estimation and Poisson rate matrix estimation problems for all unitarily invariant norms.
 
 
%
\smallskip

\end{abstract} 	



	
%
%
%
%



\section{Introduction}
	\label{sec:intro}

\subsection{Motivation} 
\label{sec:motivation}

Driven by contemporary applications such as functional genomics, network analysis, etc., 
there has been a recent surge in the study of estimating large 
mean and covariance 
matrices in the statistics community.
See, for instance, \cite{Bunea11rank,Candes11,Koltchinskii11,Negahban11,Rohde11}
and \cite{Bickel08,Bickel08thr,CZZ10,CZ12,ElKaroui08}.
From a decision-theoretic point of view, the minimax risk characterizes the fundamental limit of estimation accuracy in these problems.
When it is difficult to evaluate the exact minimax risk, as is often the case in high dimensions,
minimax rate serves as a proxy which approximates the minimax risk non-asymptotically within absolute constant factors.
The minimax rate thus captures the essential statistical difficulty of the problem and sheds light on the interplay between different parameters in the model.

Two major challenges arise from large matrix estimation problems:
\begin{enumerate}
\item 
The matrix estimand is a \emph{finite} but \emph{high dimensional} object. 
In many contexts the size of the matrix far exceeds the sample size and/or the signal-to-noise ratio. 
Furthermore, various two-dimensional structures and spectral properties render the matrix estimation problems intrinsically different from their vector (one-dimensional) counterparts. 
	
\item 
The matrix norms involved in the loss function can be different from the Frobenius norm used in the traditional quadratic loss.
For example, Bickel and Levina \cite{Bickel08,Bickel08thr} considered spectral norm loss for covariance matrix estimation; \citet{Rohde11} used Schatten norm loss in the study of trace regression.


\end{enumerate}

By the equivalence of norms on finite-dimensional spaces, characterizing the minimax rate under the usual quadratic loss (squared Frobenius norm) automatically yields lower and upper bounds for the risk under other norms. However, such soft analysis usually does not yield tight minimax rates 
that are within universal constant factors of the minimax risk 
over all model parameters non-asymptotically. 
As pointed out by \citet{CZZ10},
the minimax rates of convergence of these matrix estimation problems depend critically on the choice of norm in the loss function.
In the literature, such dependence has so far been explored in each problem mostly on a case-by-case basis.
Determining the minimax rates under general matrix norm losses calls for new constructions and machinery.

For matrix estimation, many of the commonly used norms in the loss function fall into the category of \emph{unitarily invariant norms}. 
Examples include, but are not limited to, Frobenius norm, spectral norm, and, more generally, the classes of Schatten norms and Ky Fan norms \cite{Bhatia}.
Therefore, it is of interest to develop a unified theory for all such norms.
The precise definition of unitarily invariant norms will be given in \prettyref{sec:prelim}. 
Roughly speaking, these norms are invariant under the action of the orthogonal group. 

As an attempt to address the aforementioned challenges, we aim to establish in this paper minimax rates in several matrix estimation problems for \emph{all} unitarily invariant norm losses via a \emph{unified} approach.
The classical minimax theory largely depends on the inner product structure endowed by the Frobenius norm. 
In contrast, the results of the current paper depend crucially on the geometry of the normed space, and in particular, volumes of convex bodies in finite-dimensional Banach spaces equipped with the norms of interest.

\subsection{A representative example}
\label{sec:intro-eg}
To illustrate our approach, consider the following matrix denoising problem. Suppose we observe a $p\times m$ matrix
\begin{equation}
\label{eq:sp-reg-model}
Y = M + Z,
\end{equation}
where $M$ is the unknown matrix contaminated by $Z$ with i.i.d.~$N(0,\sigma^2)$ entries.
For simplicity, assume that $\sigma=1$. In addition, we assume that $M$ has at most $k$ nonzero rows and $s$ nonzero columns, which are not necessarily consecutive.
Denote the collection of all such matrices by $\calF(k,s;p,m)$. 
We are interested estimating $M$ in the high-dimensional setting where both $p$ and $m$ can be large while $k$ and $s$ can be much smaller than $p$ and $m$.

Since the nonzeros of $M$ concentrate on a $k\times s$ submatrix, we call this structure \emph{submatrix sparsity}. 
This model arises in a number of interesting applications, \eg,
\begin{itemize}
\item It provides a concise model for studying \emph{biclustering} of microarray data. Let each row of the data matrix represents a gene and each column a patient. A subset of $s$ patients may have the same subtype of cancer and should be clustered together. Meanwhile, this cancer subtype only involves a small set of $k$ genes, which should also be identified as a cluster and at the same time linked to the $s$ patients. 
This biologically meaningful structure is well captured by the submatrix sparsity model, which, along with its variants, has been investigated in \cite{SWPN09,BI12,SN13} for this purpose.

\item When $s=m$, there is no sparsity along the columns,
and submatrix sparsity reduces to \emph{group sparsity} as a special case.
Group sparsity has been studied in the context of high-dimensional regression \cite{Yuan06,Lounici11} and has important application in multi-task learning. More recently, it has also been found useful for sparse principal component analysis \cite{CMW12}.

\item Another closely related problem is \emph{community detection} in networks. For instance, in \cite{CV13} a community is modeled as a complete (or dense) subgraph which represents itself as a submatrix in the global adjacency matrix. This is also related to the planted clique problem \cite{AKS98} in theoretical computer science.
\end{itemize}


For this problem, the techniques developed in the current paper lead to the following result.\footnote{Let $\fnorm{\cdot}$ denote the Frobenius norm, and for two sequences $\{a_n\}$ and $\{b_n\}$, we write $a_n\asymp b_n$ if for some absolute constants $0<c\leq C<\infty$, $c\leq a_n/b_n\leq C$ holds for all $n$.}

\begin{theorem}
	\label{thm:example}
Let $\|\cdot\|$ be any unitarily invariant norm on $\reals^{p\times m}$. 
Let $r = \min(k,s)$ and $I_r\in \reals^{p\times m}$ have ones on the first $r$ diagonal entries and zeros everywhere else.
Let $L_{\|\cdot\|} = \sup\{\|A\|: A\in \calF(k,s;p,m),\, \fnorm{A} =  1\}$.
The minimax rate for estimating $M$ under model \eqref{eq:sp-reg-model} is given by
\begin{equation}
	\label{eq:example}
\inf_{\widetilde{M}} \sup_{M\in \calF(k,s;p,m)} \Expect \|\widetilde{M} - M\|^2 
\asymp \|I_r\|^2(k+s) + L_{\|\cdot\|}^2\pth{k\log\frac{\eexp p}{k} + s\log\frac{\eexp m}{s}}.
\end{equation}
\end{theorem}

The significance of \prettyref{thm:example} is threefold. 
First, it determines the minimax rates of estimation simultaneously for all unitarily invariant norms.
Second, for any \ui norm, the minimax rate admits the same form as the sum of two terms. As we shall clarify later, the first term appears even when we have the oracle knowledge of the locations of the nonzero rows and columns, and is hence called the \emph{oracle risk}. The second term in the rate 
stems from the combinatorial uncertainty about the row and column support, which we refer to as the \emph{excess risk}.
Last but not least, the theorem shows that for any \ui norm, the minimax rate depends on the norm only through two quantities: 1) the norm of the $I_r$ matrix involved in the oracle risk, and 2) the (restricted) Lipschitz constant $L_{\|\cdot\|}$ of the norm contained in the excess risk term.



\subsection{Convex geometry and minimax rates}
In many matrix estimation problems, such as the denoising problem in \prettyref{sec:intro-eg},
the (matrix) parameter of interest belongs to, or can be well approximated by an element in, a linear subspace of much lower dimension than the size of the matrix. 
Further examples include
banded/bandable matrices \cite{Bickel08}, sparse matrices \cite{Bickel08thr}, low rank matrices \cite{Negahban11}, 
spiked covariance matrices \cite{CMW12}, 
among others.
For simplicity, we shall call this lower-dimensional space the \emph{support} of the parameter.
As illustrated by \prettyref{thm:example} for submatrix sparsity, it has been observed that the minimax rates of various structured problems (\eg, \cite{CMW12,Lounici11,RWY12}) can be expressed as the sum of the oracle and the excess risks, though
it is possible that one term dominates the other in certain regimes. 

As a logical step toward determining the minimax rates in structured problems, 
we
first investigate the minimax rates in the absence of structural assumptions.
This approach yields a legitimate lower bound to the corresponding structured problem via an oracle argument by assuming the additional knowledge of the support.
In addition, 
it provides us with insights on how the statistical difficulty depends on the interplay between the norm structure of the problem and the noise statistics. 
We note that these ``unstructured'' problems are not at all trivial, because our goal is to obtain the minimax rates with respect to
\emph{all} unitarily invariant norms.

The oracle lower bounds are obtained by an application of the Fano's lemma to a local Kullback-Leibler (KL) neighborhood, followed by bounding the packing number via volume estimates. Ibragimov and Has'minskii 
\cite{IKbook}
pioneered the information-theoretic technique of using Fano's inequality and metric entropy to derive minimax lower bounds, with later developments in, \eg, \cite{Birge83,HO97,Yu97,YB99}. 
The standard strategy is to turn the estimation problem into a multiple hypothesis testing problem by choosing an $\epsilon$-packing set (with respect to the loss) of the parameter space. If the log-cardinality of the set is sufficiently larger than the maximal mutual information, then the hypotheses cannot be discriminated reliably, which then incurs an estimation error no less than $\epsilon$. Capitalizing on the finite-dimensionality and the volume measure on the Euclidean space, we take this standard method one step further by lower bounding the packing number in terms of the following \emph{volume ratio}: 
\begin{equation}
\frac{\vol(\text{KL neighborhood})}{\vol(\text{norm ball})},
	\label{eq:volratio-0}
\end{equation}
which captures the interplay between the statistical structure and the metric structure. 
This abstract approach allows us to sidestep the explicit construction of packing sets used in Fano's inequality.
Exploiting the connections between Gaussian measures and volume estimates in convex geometry, we further bound the volumes of the KL neighborhood and the norm ball from below and above using Urysohn's inequality and inverse Santal\'o's inequality \cite{Pisier-book}, respectively. 
As a consequence, the \emph{Gaussian width} of the norm ball plays a key role in the oracle lower bounds.

The volume method is in fact applicable beyond the normal mean model, in which case the KL neighborhood need not coincide with an Euclidean (Frobenius) ball. For instance, the KL neighborhood for the Gaussian covariance model (resp. Poisson model) can be approximated by the intersection of a Frobenius ball and a spectral norm ball (resp. hypercube). These departures from the normal mean model yield subtle differences in the respective minimax rates. However, surprisingly, the oracle minimax rates in all three estimation problems depend on the norm only through its value at the \emph{identity matrix}.

Turning back to structured problems,
we need to further determine the excess risk,
which can depend on the norm in a very different way from the oracle risk.
In this paper, we use the mean matrix estimation with submatrix sparsity problem \eqref{eq:sp-reg-model} as the leading example to illustrate this point. 
In this problem, 
the excess risk depends on any unitarily invariant norm only through its (restricted) \emph{Lipschitz constant} with respect to the Frobenius norm.
In contrast, the oracle risk only depends on the norm of the identity matrix.
Due to tremendous freedom in imposing structural assumptions, a general theory on the excess risk is beyond the scope of the current paper.
However, the lower bound technique developed for this problem in \prettyref{sec:group-sparse} can be readily generalized to study other sparsity-constrained problems under any unitarily invariant norm losses. See, in particular, \prettyref{lem:tau-norm-dispersion}.
In addition to model \eqref{eq:sp-reg-model}, we also considered the problem of matrix completion as a second example of structured normal mean matrix estimation problem.

\subsection{Connection to the literature}


Closely related to our lower bound techniques are the celebrated minimax rate results of Yang and Barron \cite{YB99} and Birg\'e \cite{Birge83}, which are obtained for general models under conditions of the loss function as well as the metric entropy growth conditions. 
In this paper, we only impose minimal technical conditions since we focus on concrete matrix models. Moreover, we note the following distinctions which render the results from \cite{YB99} and \cite{Birge83} not directly applicable:
\begin{enumerate}
	\item \citet{YB99} gives the optimal rate for minimax estimation over massive parameter sets, whose metric entropy (with respect to the KL divergence) grows super polynomially. This applies to many infinite-dimensional function spaces such as those infinite-dimensional spaces used in nonparametric function estimation. 
	However, as pointed out in \cite[Section 7]{YB99}, their lower bound is known to be loose for finite-dimensional spaces, while the matrices of primary interest in this paper are \emph{finite-but-high-dimensional} objects. 
	\item 
	While the minimax lower bound in \cite[Theorem 1]{YB99} applies to arbitrary losses satisfying a weak triangle inequality, it was only shown to be tight for the KL loss $L(\theta,\theta') = \KL{P_\theta}{P_{\theta'}}$ or its equivalent under suitable entropy growth conditions. On the other hand, the results in \cite{Birge83} are dedicated to squared Hellinger loss. 
	In contrast, our method is applicable to any norm loss under the matrix models considered in the current paper, and, in particular, optimal for all unitarily invariant norm losses.
\end{enumerate}

The main results in this paper deal with loss functions that are invariant under the action of the orthogonal group.
The significance of invariant decision problems have long been recognized in the statistics literature.
They played a crucial role in understanding the relationship between invariant estimators and minimaxity (the Hunt-Stein theorem) as well as that between shrinkage estimators and orthogonally equivariant estimators \cite{Beran96}. 
Group-invariant losses have been considered 
by Stein \cite{Stein56}, Eaton \cite{Eaton70, Eaton89}, 
etc. in covariance matrix estimation problems in low dimensions, though the emphasis therein is on exact minimax risks rather than the rates.

Besides matrix estimation, the minimax inference under non-quadratic losses has been considered in various vector estimation problems as well. For instance, Donoho and Johnstone \cite{DJ94} studied the sharp asymptotics of the minimax risk for estimating an unknown mean vector in an $\ell_p$-ball under the $\ell_q$-norm loss in the Gaussian sequence model. 


\subsection{Contribution and paper organization}
The main contribution of the current paper is as follows:
\begin{enumerate}
\item 	We develop a new approach for establishing minimax lower bounds in matrix estimation problems for all squared unitarily invariant norm losses.
	The approach does not require explicit construction of the least favorable configuration within the parameter space of interest.

\item We determine the minimax rates with respect to all squared unitarily invariant norm losses for estimating Gaussian mean matrices under submatrix sparsity. This includes group sparsity as a special case. 
We show that the two terms in the minimax rates depend on the choice of norm in completely different ways.

\item We establish minimax lower bounds for the matrix completion problem with respect to all \ui norms. Our lower bounds show that the estimator developed in \citet{Koltchinskii11} achieves near optimal rates for all squared Schatten-$q$ norm losses with $q\in [1,2]$. 
This is among the few optimality results for matrix completion beyond the squared Frobenius norm loss.

\item We show that the new machinery works beyond normal mean matrix estimation settings, where covariance matrix estimation and Poisson rate matrix estimation serve as leading examples.
\end{enumerate}

The rest of the paper is organized as follows.
\prettyref{sec:prelim} introduces notations and preliminary results on unitarily invariant norms and volume of convex bodies.
In \prettyref{sec:oracle-theta}, we use the volume approach to study the oracle minimax rates in mean matrix estimation.
\prettyref{sec:group-sparse} investigates the minimax estimation of mean matrices under two kinds of structural constraints, namely submatrix sparsity and low-rankness.
\prettyref{sec:beyond} presents two examples beyond normal mean matrix estimation where our machinery yields tight rates.
We conclude with discussion in \prettyref{sec:discuss}. 
Further technical details are included in the appendix.


\section{Preliminaries}
\label{sec:prelim}

In this section, we introduce the basic notation, give the definition of and some preliminary facts about unitarily invariant norms, and review several existing results in the literature on volume ratios of convex bodies that will be useful for our lower bound construction.

\paragraph{Notation}
For any matrix $\bfX = (x_{ij})$, 
the $i\Th$ row of $\bfX$ is denoted by $\row{\X}{i}$ and the $j\Th$ column by $\col{\X}{j}$.
For a positive integer $p$, $[p]$ denotes the index set $\{1, 2, ..., p\}$. For any set $I$, $|I|$ denotes its cardinality and $\comp{I}$ its complement. 
For two subsets $I$ and $J$ of indices, we write $\X_{IJ}$ for the $|I|\times |J|$ submatrices formed by $x_{ij}$ with $(i,j) \in I \times J$.
When $I$ or $J$ is the whole set, we abbreviate it with a $*$, and so if $X\in \reals^{n\times p}$, then $\row{\X}{I} = \X_{I [p]}$ and $\col{\X}{J} = \X_{[n] J}$.
For any square matrix $\AA = (a_{ij})$, denote its trace by $\tr(\AA) = \sum_{i}a_{ii}$. 
Denote by $\sfS_k$ (resp. $\sfS_k^+$) the set of $k\times k$ symmetric (resp. positive semi-definite) matrices.
Moreover, let $O(k)$ denote the set of all $k\times k$ orthogonal matrices.
For any matrix $\AA \in \reals^{k \times s}$, $\sigma_i(\AA)$ stands for its $i\Th$ largest singular value and $\sigma(\AA) = (\sigma_1(\AA), \ldots, \sigma_{k \wedge s}(\AA))'$ the vector of ordered singular values.
When $\AA \in \sfS_k^+$, $\sigma_i(\AA)$ is also the $i\Th$ largest eigenvalue of $\AA$.
We use $\ones_d$ to denote the all-one vector in $\reals^d$, though the dependence on $d$ might be dropped when there is no ambiguity.

For any real number $a$ and $b$, set $a\vee b = \max\{a,b\}$, $a\wedge b = \min\{a,b\}$ and $a_+ = a\vee 0$.
For any sequences $\{a_n\}$ and $\{b_n\}$ of positive numbers, we write $a_n \gtrsim b_n$ if $a_n\geq cb_n$ holds for all $n$ and some absolute constant $c > 0$, $a_n\lesssim b_n$ if $b_n \gtrsim a_n$, and $a_n \asymp b_n$ if both $a_n\gtrsim b_n$ and $a_n\lesssim b_n$ hold.


\subsection{Unitarily invariant norms}
\label{sec:norm}

We refer to \cite[Sections 5.1 and 5.6]{horn} for the defining properties of vector and matrix norms.
On an inner product space, the \emph{dual norm} of a norm $\|\cdot\|$ is defined as 
\begin{equation}
\|x\|_* = \sup_{\|y\| \leq 1} \iprod{x}{y}.
	\label{eq:dualnorm}
\end{equation}
In this paper, we shall encounter two standard inner product spaces: 
1) the Euclidean space $\reals^d$ with the usual inner product $\iprod{x}{y}=x'y$, and 
2) the space of $k \times s$ matrices, denoted by $\reals^{k \times s}$, with inner product $\iprod{\bfA}{\bfB} = \tr(\bfA'\bfB)$.
The latter inner product can be reduced to the former if we vectorize both $\bfA$ and $\bfB$ by stacking their columns into vectors in $\reals^{k s}$.
By definition, we have the duality result:  $\iprod{x}{y} \leq \|x\|  \|y\|_*$. 

To define unitarily invariant norms, we first introduce the notion of symmetric gauges.
A function $\tau: \reals^d\to [0,\infty)$ is called a \emph{symmetric gauge function} (or a $1$-symmetric norm)
if it is a norm on $\reals^d$ which is invariant with respect to sign changes and permutations \cite{horn}.
That is, for any $x\in \reals^d$, $\tau(\ntok{\epsilon_1 x_{\pi(1)}}{\epsilon_d x_{\pi(d)}})=\tau(\ntok{x_1}{x_d})$ for any permutation $\pi$ on $[d]$ and any $\epsilon=(\ntok{\epsilon_1}{\epsilon_d})\in \{-1,1\}^d$. 
 The following lemma summarizes two properties of symmetric gauges which we use frequently in the rest of the paper.
Its proof is given in \prettyref{app:pf-lemmas}.
 \begin{lemma}
Let $\tau$ be a symmetric gauge function on $\reals^d$. Then 
\begin{enumerate}
	\item $\tau$ is monotone:
	$\tau(x_1, x_2,\dots, x_d) \geq \tau(x_1', x_2,\dots, x_d)$ for any $|x_1| \geq |x_1'|$ and any $x_2,\dots,x_d$;
	\item The dual norm $\tau_*$ is also a symmetric gauge function and satisfies $\tau_*(\ones) \tau(\ones) = d$. 
\end{enumerate} 
	\label{lmm:tau}
\end{lemma}


A matrix norm $\|\cdot\|$ is called a \emph{unitarily invariant norm} if for any $\A\in \reals^{k \times s}$ and any orthogonal matrices $U\in O(k)$ and $V\in O(s)$, $\|\A\| = \|U\A V\|$.
Recall that $\sigma(A)$ is the vector in $\reals^{k\wedge s}$ consisting of the singular values of $A$.
A fundamental result due to von Neumann \cite{vN37} states that for any unitarily invariant norm $\|\cdot\|$ on $\reals^{k \times s}$, there exists a symmetric gauge function $\tau$ on $\reals^{k \wedge s}$ such that
\begin{equation}
	\|\A\| = \tau(\sigma(\A)).
	\label{eq:taunorm}
\end{equation}
Henceforth we denote the unitarily invariant norm \prettyref{eq:taunorm} by $\|\cdot\|_{\tau}$. Therefore, $\tau$ and $\|\cdot\|_\tau$ are explicitly related through $\tau(x)=\|\diag(x)\|_\tau$, where $\diag(x)$ is a diagonal matrix with the elements of $x$ on the diagonal.
On the space of $k \times s$ matrices, the dual norm of a unitarily invariant norm $\|\cdot\|_{\tau}$ is $\|\cdot\|_{\tau_*}$ \cite[Proposition IV.2.11]{Bhatia}, where $\tau_*$ is the dual norm of $\tau$ on $\reals^{k \wedge s}$.

Let $\|\cdot\|$ be a norm on $\reals^d$ and $\|\cdot\|_2$ denote the Euclidean norm. 
Note that all norms are equivalent in a finite-dimensional space.
Thus, for the mapping $x \mapsto \norm{x}$, its Lipschitz constant (with respect to the Euclidean norm)
\begin{equation}
L_{\|\cdot\|} = \sup_{\x\neq y} \frac{|\|x\| - \|y\||}{\|x-y\|_2}
= \sup_{x \neq 0 } \frac{\|x\|}{\,\,\|x\|_2}
	\label{eq:Ltau}
\end{equation}
is finite. To see the last equality, note that the first supremum is greater than the second by taking $y=0$, while the other direction follows from the triangle inequality $|\|x\| - \|y\||\leq \|x-y\|$.
The Lipschitz constant of any matrix norm is defined as
\begin{equation}
L_{\|\cdot\|} =   \sup_{A\neq 0} \frac{\|A\|}{\,\,\,\fnorm{A}}.
	\label{eq:Lip}
\end{equation}
For any \ui norm $\|\cdot\|_\tau$, it is straightforward to verify that
\begin{equation}
	L_{\|\cdot\|_\tau} = L_\tau, 
	\label{eq:Ltau-1}
\end{equation}
where $L_\tau$ is the Lipschitz constant of $\tau$ as a vector norm.
We note that the following bound for $L_\tau$:
\begin{equation}
\label{eq:Ltau-tau1}
\frac{\tau(\ones)}{\sqrt{d}} \leq L_\tau \leq \tau(\ones),
\end{equation}
where the left inequality follows from \eqref{eq:Ltau} with $x = \ones$, and the right inequality is due to the following: For any nonzero vector $x$ and any symmetric gauge function $\tau$, $\tau(\ones) \|x\|_2 = \tau(\|x\|_2,\dots, \|x\|_2) \geq \tau(x_1,\dots, x_d) = \tau(x)$, in view of the monotonicity of $\tau$ in \prettyref{lmm:tau}. 

Two important classes of \ui norms are
{Schatten norms} and {Ky Fan norms}.
For any $q\in [1,\infty]$, the Schatten $q$-norm of 
$\bfA = (a_{ij}) \in \reals^{k \times s}$ is
\begin{equation}
\label{eq:sq-norm}
\lsqnorm{\bfA} = \pth{\sum_{i=1}^{k \wedge s} \sigma_i^q(\bfA)}^{1/q}.
\end{equation}
The dual norm of $\lsqnorm{\cdot}$ is $\lsqnormdual{\cdot}$, where $\frac{1}{q}+\frac{1}{q^*}=1$.
For any $\ell \in [k \wedge s]$, the Ky Fan $\ell$-norm of $\bfA$ is
\begin{equation}
	\label{eq:kf-norm}
	\kfnorm{\bfA}{\ell} = \sum_{i=1}^{\ell} \sigma_i(\bfA),
\end{equation}
whose dual norm
is $\max\{\snorm{\infty}{\cdot}, \ell^{-1}\snorm{1}{\cdot} \}$ \cite[p.96]{Bhatia}.
The Lipschitz constants of the Schatten-$q$ norm and the Ky-Fan $\ell$-norm are
\begin{equation}
	L_{{\rm S}_q} = r^{(1/q-1/2)_+} 
	\qquad \mbox{and} \qquad
	L_{(\ell)} = \sqrt{\ell}.
		\label{eq:Lsq-Lkf}
\end{equation}

Note several special cases:
1) Frobenius norm: $\snorm{2}{\bfA} = (\sum_i \sigma_i^2(A))^{1/2} = (\sum_{i,j} a_{ij}^2)^{1/2}$, also denoted by $\fnorm{A}$;
2) Spectral (operator) norm: $\snorm{\infty}{\bfA} = \kfnorm{\bfA}{1} = \sigma_1(A)$, also denoted by $\opnorm{A}$;
3) Nuclear norm: $\snorm{1}{\bfA} = \kfnorm{\bfA}{k \wedge s} = \sum_{i=1}^{k \wedge s} \sigma_i(\bfA)$.

\subsection{Volume ratio of convex bodies}
	\label{sec:volume}
	
	
We now introduce a few useful results on volume ratios of convex bodies in finite-dimensional Banach spaces. 

In this paper, we focus on two specific finite-dimensional spaces: the space $\reals^{k \times s}$ of $k \times s$ matrices and the space $\sfS_k$ of $k\times k$ symmetric matrices. 
(Either of them can be equipped with a variety of different norms depending on the context though.)
In both spaces, the volume of any compact set $K$ is given by $\vol(K) = \int_K \diff M$, where $\diff M$ denotes the volume elements, defined as follows respectively:
The volume element of $\reals^{k\times s}$ is the usual Lebesgue measure $\diff M = \prod_{i,j} \diff m_{ij}$.
For $\sfS_k$, which is a linear subspace of $\reals^{k\times k}$ due to the symmetry constraint, its volume element is 
$\diff M = 2^{\frac{k(k-1)}{4}} \prod_{i \in [k]}\diff m_{ii}	\prod_{1\leq i < j\leq k}\diff m_{ij}$,
by the Jacobian formula. 


Recall that $K$ is a symmetric convex body in $\reals^d$ if $K$ is a compact convex set with non-empty interior such that $K = - K$. 
The most commonly encountered symmetric convex bodies are norm balls, for which we introduce the following notations: Let $B_{\norm{\cdot}}^d(\epsilon)=\{x \in \reals^d: \norm{x} \leq \epsilon\}$ denote the norm ball of radius $\epsilon$ centered at zero. 
Let $B_2^d$ and $B_2^{k \times s}$ denote the unit Euclidean ball and Frobenius ball at zero in $\reals^d$ and $\reals^{k \times s}$, respectively. 
We sometimes omit the dimension in the superscript when no confusion arises. 

The polar of a convex body $K$ is defined as follows
\begin{equation}
K^{\circ} = \Big\{y \in \reals^d: \sup_{x \in K} \iprod{x}{y} \leq 1 \Big\},
	\label{eq:polar}
\end{equation}
which is also a convex body. The Minkowski functional of a symmetric convex body $K$ is defined as
\begin{equation}
	\norm{x}_K = \inf\{r > 0: x \in r K\},
	\label{eq:normK}
\end{equation}
also known as the gauge of $K$. 
If $K=\{x: \|x\| \leq 1\}$ is some unit norm ball, then $\|\cdot\|_K = \|\cdot\|$.



The following inequality due to Urysohn \cite{Urysohn24} (see also \cite[p.~7]{Pisier-book}) reveals a deep connection between the volume ratio of a convex body $K$ and the Gaussian measure:
\begin{lemma}[Urysohn's Inequality]
Let $K$ be a symmetric convex body in $\reals^d$. Then
\begin{equation}
	\pth{\frac{\vol(K)}{\vol(B_2^d)}}^{\frac{1}{d}} \leq \frac{1}{\sqrt{d}}\,\expect{\sup_{y \in K} \iprod{G}{y}},
	\label{eq:urysohn}
\end{equation}	
where $G \sim N(0,\bfI_d)$ is standard Gaussian. The expectation of the supremum on the right-hand side of \prettyref{eq:urysohn} is called the \emph{Gaussian width} of $K$.
	\label{lmm:urysohn}
\end{lemma}
Moreover, for any symmetric convex body $K \subset \reals^d$,
\begin{equation}
\frac{1}{2}\leq \pth{\frac{\vol(K) \vol(K^{\circ})}{\vol(B_2^d)^2}}^{\frac{1}{d}} \leq 1.
	\label{eq:BM}
\end{equation}
The upper bound is known as Santal\'o's inequality \cite[p. 100]{Pisier-book}. The lower bound is first proved by Bourgain and Milman (see, \eg, \cite[Corollary 7.2]{Pisier-book}) for some universal constant $\alpha>0$, and the specific value of $\frac{1}{2}$ is shown by Kupenberg \cite{Kuperberg08}.
In view of \prettyref{eq:BM} and the fact that 
\begin{equation}
\vol(B_2^d)^{\frac{1}{d}} = \frac{\sqrt{\pi}}{\Gamma(\frac{d}{2}+1)^{\frac{1}{d}}} \asymp \frac{1}{\sqrt{d}},	
	\label{eq:vol-B2}
\end{equation}
applying \prettyref{lmm:urysohn} to the polar $K^\circ$ yields the following inverse Santal\'o's inequality which is useful in lower bounding the volume of a convex body. The version here can also be found in \cite[p.92, display 4]{GP07}.
\begin{lemma}[Inverse Santal\'o's inequality]
	There exists a universal constant $c_0$, such that for any symmetric convex body $K$ in $\reals^d$,
	\begin{equation}
	\vol(K)^{\frac{1}{d}} \geq \frac{c_0}{\expect{\norm{G}_K}}.
	\label{eq:santalo}
\end{equation}
	\label{lmm:santalo}
\end{lemma}

For the space of $k \times s$ matrices, Lemmas \ref{lmm:urysohn} and \ref{lmm:santalo} hold with $d= k s$. In order to deal with the space of $k\times k$ symmetric matrices, we have the following useful generalization: Let $E \subset \reals^d$ be a linear subspace with dimension $d_E$. Let $P_E$ denote the orthogonal projection from $\reals^n$ onto $E$. Let $G_E \triangleq P_E(G)$ is the Gaussian ensemble on $E$. Then we have the following generalization of Lemmas \ref{lmm:urysohn} and \ref{lmm:santalo}:
\begin{equation}
\pth{\frac{\vol(P_E(K))}{\vol(P_E(B_2))}}^{\frac{1}{d_E}} \leq \frac{1}{\sqrt{d_E}}\,\expect{\sup_{y \in K} \iprod{G_E}{y}},
	\label{eq:urysohn-E}
\end{equation}	
and
	\begin{equation}
	\vol(P_E(K)^\circ)^{\frac{1}{d_E}} \geq \frac{c_0}{\expect{\norm{G_E}_K}},
	\label{eq:santalo-E}
\end{equation}
where $\vol(\cdot)$ is with respect to the volume element on the subspace $E$. Note that the polar $P_E(K)^\circ$ is defined in the subspace $E$ and we have $P_E(K)^\circ = K^\circ \cap E$ \cite[Proposition 2., p.9]{Vershynin-GFA}.
Note that $\sfS_k$ is a subspace of $\reals^{k \times k}$ with dimension $d= \frac{1}{2}k(k+1)$, with orthogonal projection $A \mapsto \frac{A+A'}{2}$. Then $P_{\sfS_k}(B_2) = B_2 \cap \sfS_k$ and $G_{\sfS_k} = \frac{G+G'}{2}$, which coincides with the Gaussian orthogonal ensemble GOE($k$).

\section{Volume ratio and unconstrained mean matrix estimation}
\label{sec:oracle-theta}

As we have mentioned in the introduction, understanding the minimax rates for unconstrained matrix estimation is the first step toward deriving the rates in those with structural constraints.
In this section, we derive tight minimax rates for estimating unconstrained mean matrices under all unitarily invariant norms.

In model \eqref{eq:sp-reg-model}, if we are informed with the knowledge of the support by an oracle, the problem reduces to the following unconstrained version where we observe the $k\times s$ matrix 
\begin{equation}
\label{eq:additive-noise-model}
Y = M + Z,
\end{equation}
where $M\in \reals^{k\times s}$ is the matrix to be estimated, and $Z = (z_{ij})$ is the noise matrix with i.i.d.~$N(0,1)$ entries.
When $z_{ij}$ are \iid~$N(0,\sigma^2)$, our results continue to hold after multiplied by an extra factor of $\sigma^2$.

\subsection{Volume ratio, Gaussian width, and a general lower bound}

Note that we can always vectorize the $Y, M$ and $Z$ matrices in \eqref{eq:additive-noise-model}, and the problem then reduces to a $d$-dimensional Gaussian mean problem with $d = ks$.
In addition, any matrix norm on $\reals^{k\times s}$ induces a vector norm on $\reals^{d}$.
In view of this connection, we derive below a general lower bound for estimating a $d$-dimensional vector in Gaussian white noise.

To this end, we first establish the connection between minimax lower bounds and volume ratios in the following proposition, which is a slight variant of Fano's lemma \cite[Lemma 5.1, p.356]{IKbook} (see also \cite[Proposition 2.8]{Birge83} and \cite[Section 2.7.1]{Tsybakov09}).

\begin{proposition}
	Let $(\Theta, \rho)$ be a metric space and $\{P_\theta: \theta \in \Theta\}$ a collection of probability measures. 	For any totally bounded $T \subset \Theta$, denote by $\calM(T, \rho,\epsilon)$ the $\epsilon$-packing number of $T$ with respect to $\rho$, \ie, the maximal number of points in $T$ whose pairwise minimum distance in $\rho$ is at least $\epsilon$.
Define the \emph{Kullback-Leibler diameter} of $T$ by
	\begin{equation}
	\dKL(T) \triangleq \sup_{\theta,\theta' \in T} \KL{P_\theta}{P_{\theta'}}.
	\label{eq:dKL}
\end{equation}
	 Then
	\begin{equation}
	\inf_{\hat{\theta}} \sup_{\theta \in \Theta} \Expect_{\theta}[\rho^2(\hat{\theta}(X),\theta)] \geq \sup_{T \subset \Theta} \sup_{\epsilon > 0} \frac{\epsilon^2}{4} \pth{1 -  
	\frac{\dKL(T) + \log 2}{\log \calM(T,\rho,\epsilon)}}.
	\label{eq:fano}
\end{equation}
In particular, if $\Theta \subset \reals^d$ and 
$\norm{\cdot}$ is some norm on $\reals^d$, then 
\begin{equation}
	\inf_{\hat{\theta}} \sup_{\theta \in \Theta} \Expect_{\theta}[\|\hat{\theta}(X) - \theta\|^2] \geq \sup_{T \subset \Theta} \sup_{\epsilon > 0} \frac{\epsilon^2}{4} \pth{1 -  
	\frac{\dKL(T) + \log 2}{\log \frac{\vol(T)}{\vol(B_{\|\cdot\|}(\epsilon))}}}.
	\label{eq:fano-vol}
\end{equation}
	\label{prop:fano}
\end{proposition}

\begin{remark}

The minimax lower bound obtained via the global entropy method \cite{YB99} amounts to choosing $T = \Theta$ (or a compact set thereof with constant KL diameter) on the right-hand side of \prettyref{eq:fano}. This method is usually most useful in infinite-dimensional space. In contrast, in finite-dimensional space, local entropy method gives tight lower bound when we use $T$ whose KL diameter is on the order of $\frac{1}{n} \times \text{dimension}$. See also the discussion in \cite{Guntuboyina11}. The method of local metric entropy dates back to Le Cam \cite{LeCam73}. 
	\label{rmk:local-global}
\end{remark}

\begin{proof}
Let $\{\theta_i: i\in[N]\} \subset T$ be a maximal $\epsilon$-packing set, where $N = \calM(T, \rho,\epsilon)$ and $\min_{i \neq j} \rho(\theta_i,\theta_j) \geq \epsilon$.
Applying Fano's lemma, the average probability of error for the multiple hypothesis testing problem $\{P_{\theta_i}: i \in [N]\}$ is lower bounded by
\[
p_e \geq 1 - \frac{\min_{i \neq j} D(P_{\theta_i} || P_{\theta_j}) + \log 2}{\log \calM(T,\rho,\epsilon)}.
\]
The estimation lower bound \prettyref{eq:fano} is obtained by applying triangle inequality. 

The lower bound \prettyref{eq:fano-vol} is obtained by bounding the packing number from below by the volume ratio: Denote by $\calN(T, \|\cdot\|,\epsilon)$ the $\epsilon$-covering number of $T$ with respect to the norm $\|\cdot\|$, \ie, the minimal number of balls of radius $\epsilon$ whose union contains $T$. 
Then 
$\calM(T,\|\cdot\|,\epsilon) \geq \calN(T,\|\cdot\|,\epsilon)$ \cite[Theorem IV]{KT59}. In view of the translation invariance of the volume measure, applying the union bound yields $\calN(T,\|\cdot\|,\epsilon) \geq \frac{\vol(T)}{\vol(B_{\|\cdot\|}(\epsilon))}$, completing the proof.
\end{proof}

\begin{remark}
The proof of \prettyref{prop:fano} in fact establishes the following high-probability lower bound: For any $\epsilon > 0$, 
	\begin{equation}
	\inf_{\hat{\theta}} \sup_{\theta \in \Theta} \Prob_{\theta}(\|\hat{\theta}(X) - \theta\| \geq \epsilon/2) \geq 1 -  \inf_{T \subset \Theta}
	\frac{\dKL(T) + \log 2}{\log \frac{\vol(T)}{\vol(B_{\|\cdot\|}(\epsilon))}}.
	\label{eq:fano-vol-highprob}
\end{equation}
	\label{rmk:highprob}
\end{remark}

The specialization of \prettyref{prop:fano} to Gaussian measures, 
together with \prettyref{lmm:urysohn}, leads to the following result for Gaussian location model. 
\begin{theorem}[General norm]
Let $d \in \naturals$. Consider the Gaussian location model 
$Y = \theta + Z$, 
where $\theta \in \reals^d$ and $Z \sim N(0,\bfI_d)$ is a $d$-dimensional white noise vector. 
Then there exists a universal constant $c_1 \in (0,1)$, such that for any $d$ and any norm $\|\cdot\|$ on $\reals^d$,
\begin{equation}
  \frac{c_1 d^2}{(\expect{\|Z\|_*})^2} \leq \inf_{\hat{\theta}} \sup_{\theta \in \reals^d} \Expect_{\theta} \|\hat{\theta}(Y) - \theta\|^2 \leq \expect{\|Z\|^2},
\label{eq:lb-dual}
\end{equation}
where $\|\cdot\|_*$ is the dual norm of $\|\cdot\|$.
	\label{thm:GLM}
\end{theorem}
\begin{remark}
\label{rmk:gaussian-width}
Recall from \prettyref{lmm:urysohn} that the Gaussian width of a symmetric convex body $K \subset \reals^d$ is $\expect{\max_{x\in K} \iprod{x}{Z}}$.
By the definition of the dual norm,
the quantity $\expect{\|Z\|_*}$ in the lower bound \eqref{eq:lb-dual} is equal to the Gaussian width of the unit ball in $\reals^d$ equipped with the norm $\|\cdot\|$ used in the loss function.
\end{remark}

\begin{proof}
The upper bound is obtained by taking the specific estimator $\hat{\theta} = Y$ and the triangle inequality. To prove the lower bound, note that the Kullback-Leibler divergence of the normal mean model is given by
\begin{equation}
	D(N(\theta,\bfI_d) \, || \, N(\theta',\bfI_d)) = \frac{1}{2} \|\theta - \theta'\|_2^2,
	\label{eq:KL-GLM}
\end{equation}
where $\|\cdot\|_2$ denotes the $\ell_2$-norm on $\reals^d$. 
Let $T = B_2(\delta) = \{\theta \in \reals^d: \norm{\theta}_2 \leq \delta\}$ denote the Euclidean ball of radius $\delta$ centered at the origin. Then $\dKL(T) \leq 4 \delta^2$. Moreover,
\begin{equation}
\frac{\vol(B_2(\delta))}{\vol(B_{\|\cdot\|}(\epsilon))} = \frac{\delta^d \vol(B_2(1))}{\epsilon^d \vol(B_{\|\cdot\|}(1))}  \geq 
\pth{\frac{\delta \sqrt{d} }{\epsilon  \, \expect{\|Z\|_*}} }^d 	,
	\label{eq:volratio}
\end{equation}
 where the last inequality follows from \prettyref{lmm:urysohn}.
Now we choose $\delta = \sqrt{d a}$ and $\epsilon = \frac{\delta \sqrt{d b}}{\expect{\|Z\|_*}} = \frac{d \sqrt{a b}}{\expect{\|Z\|_*}}$, where $a>0$ and $b\in(0,1)$ are to be optimized. Applying \prettyref{prop:fano} yields the following lower bound
\[
\inf_{\hat{\theta}} \sup_{\theta \in \Theta} \Expect_{\theta}\|\hat{\theta}(Y) - \theta\|^2 \geq \frac{c_d d^2}{(\expect{\|Z\|_*})^2},
\]
where
\begin{equation}
c_d \triangleq \sup_{0 < b < 1} \sup_{a > 0} \frac{a b}{4} \pth{1-\frac{da + 2 \log 2}{d \log \frac{1}{b}}} > 0.	
	\label{eq:cd}
\end{equation}
The proof is completed upon noting that $d \mapsto c_d$ is increasing and $d \geq 1$.
\end{proof}

\begin{remark}
It is straightforward to verify numerically that the constant $c_1$ in \prettyref{thm:GLM} satisfies $c_1 > \frac{1}{400}$. If one allows the constant to depend on the ambient dimension, then we can replace $c_1$ by $c_d$ defined in \prettyref{eq:cd}, which satisfies $\lim_{d\to \infty} c_d = \frac{1}{16 \eexp}$ in the high-dimensional setting.
	\label{rmk:constants}
\end{remark}

\begin{remark}
As an aside before proceeding to the matrix case, we note that an application	of \prettyref{thm:GLM} yields the minimax rate of the Gaussian sequence model under the squared $\ell_q$-loss:
\begin{equation}
  \inf_{\hat{\theta}} \sup_{\theta \in \reals^d} \Expect_{\theta} \|\hat{\theta}(Y) - \theta\|_{\ell_q}^2 \asymp d^{{2}/{q}},
\end{equation}
where $q \in (1,\infty)$. This follows from \prettyref{eq:lb-dual} by noting that the dual of the $\ell_q$-norm is the $\ell_{q^*}$-norm with $\frac{1}{q}+\frac{1}{q^*}=1$.
	\label{rmk:vector-lq}
\end{remark}

\subsection{Minimax rates for unitarily invariant norms}

Turning back to the matrix Gaussian location model \eqref{eq:additive-noise-model}, we are now in the position of establishing the minimax rates for estimating $M$ with respect to all unitarily invariant norms. 

Note that any matrix norm on the space $\reals^{k\times s}$ induces a vector norm on $\reals^{d}$ for $d = ks$. 
In view of \prettyref{thm:GLM}, it suffices to upper bound both $\expect{\|Z\|_*}$ and $\expect{\|Z\|^2}$, provided that the resulting lower and upper bounds agree up to a constant factor. 
It turns out that this can indeed be achieved, resulting in the following theorem.

\begin{theorem}
Let $k,s \in \naturals$ and $\|\cdot\|_\tau$ be a unitarily invariant norm, where $\tau$ is a symmetric gauge function on $\reals^{k \wedge s}$.
The minimax rate for estimating $M$ under \eqref{eq:additive-noise-model} with respect to the loss $\|\cdot\|_\tau^2$ satisfies
	\begin{equation}
\inf_{\wt{M}}\sup_{M \in \reals^{k \times s}} \Expect{\|\wt{M} - M\|_\tau^2} \asymp (k \vee s) \tau^2(\ones)
	\label{eq:oracle-theta-tau}
\end{equation}
where $\ones$ denotes the all-one vector in $\reals^{k \wedge s}$. 
	\label{thm:oracle-theta-tau}
\end{theorem}

\begin{remark}[Dependence on $\tau$]
\prettyref{thm:oracle-theta-tau} reveals the following remarkable fact: The minimax rate under the unitarily invariant norm $\|\cdot\|_\tau$ depends on the symmetric gauge function $\tau$ \emph{only} through its value at the all-one vector. 
On the one hand, $\tau(\ones)$ appears in the lower bound because it governs the volume asymptotics of a unit ball under the $\norm{\cdot}_\tau$ norm in $\reals^{k\times s}$. 
On the other hand, since the noise matrix has i.i.d.~entries,
all of its singular values scale with the dimensions at the same rate.
Hence, the risk achieved by the observation is also proportional to $\tau(\ones)$. 	
In addition, such a dependence pattern also suggests that the least-favorable prior on $M$ should concentrate on those matrices in general position, \ie, having full rank and bounded condition number.
This is intuitively natural because neither the unitarily invariant norm nor the noise singular value spectrum favor any specific direction.
	\label{rmk:tau1}
\end{remark}

\begin{remark}
	\label{rmk:MLE}
\prettyref{thm:oracle-theta-tau} also provides a rigorous justification of the following intuitive fact: If both the noise and the loss function is sufficiently symmetric, then there is nothing significantly better than estimating by the raw observation, which is the maximum likelihood estimator under the Gaussian assumption. 
Of course, such a claim crucially depends on the choice of the loss function. For example, if the loss function is given by $L(\wh{M}, M) = \rho(\fnorm{\wh{M}-M})$, where $\rho(x) = x^2 + (k\vee s)^4 \indc{x\leq 1}$, then estimating by the observation is clearly rate-suboptimal. Instead, the minimax estimator can be obtained by shrinkage towards zero.
\end{remark}

\begin{proof}[Proof of \prettyref{thm:oracle-theta-tau}]
Note that $\tau_*$ is also a symmetric gauge function.
By the monotonicity of symmetric gauge functions (cf.~\prettyref{lmm:tau}), we have for $\eta = \tau$ or $\tau_*$,
		\begin{align}
		\|\Z\|_{\eta} = \eta(\sigma(\Z)) \leq \eta(\sigma_1(\Z) \ones) = \sigma_1(\Z) \eta(\ones).
		\label{eq:Ztau}
	\end{align}

For the lower bound, \eqref{eq:Ztau} leads to
\begin{align}
\|\Z\|_{\tau_*} \leq \sigma_1(\Z) \tau_*(\ones)	= \frac{\sigma_1(\Z) (k \wedge s)}{\tau(\ones)} \,,
\end{align}
where the last equality is due to the second claim of \prettyref{lmm:tau}. Applying \prettyref{thm:GLM} yields
\begin{align*}
\inf_{\wt{M}}\sup_{M \in \reals^{k \times s}} \Expect{\|\wt{M} - M\|^2} & \geq \frac{c_1 k^2 m^2}{(\expect{\|\Z\|_{\tau_*}})^2} \geq \frac{c_1 (k \vee s)^2 \tau^2(\ones)}{(\expect{\sigma_1(\Z)})^2}\\
& \geq \frac{c_1 (k \vee s)^2 \tau^2(\ones)}{(\sqrt{k} + \sqrt{s})^2} \geq c_1 (k \vee s) \tau^2(\ones),
\end{align*}
where we have used Gordon's inequality $\expect{\sigma_1(\Z)} \leq \sqrt{k} + \sqrt{s}$; cf.~\cite{Davidson01}.

For the upper bound, in view of \eqref{eq:Ztau}, it suffices to bound $\expect{\sigma_1(Z)^2}$. 
To this end, note that the Davidson--Szarek bound \cite{Davidson01} implies that for any $a > 1$,
\[
\Prob(\sigma_1(Z) > a(\sqrt{k}+\sqrt{s})) \leq \eexp^{-(a-1)^2(\sqrt{k}+\sqrt{s})^2/2} \triangleq h(a).
\]
Together with \prettyref{lmm:power-rate} in \prettyref{app:pf-lemmas}, the last inequality implies
\begin{align}
\expect{\sigma_1(Z)^2} & \leq (\sqrt{k}+\sqrt{s})^2 \pth{1 + 2\int_1^\infty a h(a)\diff a} \nonumber \\
& = (\sqrt{k}+\sqrt{s})^2 \pth{1 + \frac{\sqrt{2\pi}}{\sqrt{k}+\sqrt{s}} + \frac{2}{(\sqrt{k}+\sqrt{s})^2}} \nonumber \\
& \leq (6 + 2\sqrt{2\pi}) (k \vee s),\label{eq:Z-opnorm2}
\end{align}
where the last inequality holds for all $k,s\geq 1$.
Applying \prettyref{thm:GLM}, together \eqref{eq:Z-opnorm2} with \eqref{eq:Ztau}, leads to the desired upper bound.
\end{proof}

\begin{remark}[Universality of the upper bound]
The rate $(k\vee s)\tau^2(1)$ in the upper bound in \prettyref{thm:oracle-theta-tau} holds under fairly general conditions.
Based on the universality results in \cite{Seginer00}, if the noise matrix $Z$ in \eqref{eq:additive-noise-model} has i.i.d.~entries, then the upper bound in \eqref{eq:oracle-theta-tau} can be established for any noise distribution with mean zero and finite fourth moment, where the constant $C$ depends only on the fourth moment $\expect{Z_{ij}^4}$. 

We now lay out a brief proof of this fact.
First, \cite[Corollary 2.2]{Seginer00} and \eqref{eq:Ztau} jointly lead to
\begin{equation}
	\label{eq:reg-oracle-upp-1}
\expect{\|\Z\|_{\tau}^2} \leq  \tau^2(\ones) \expect{[\sigma_1(\Z)^2]} \leq 
C_1 \tau^2(\ones) (\expect{\max_{i} \Fnorm{Z_{i*}}^2} + \expect{\max_{j} \Fnorm{Z_{*j}}^2}),
\end{equation}
where $C_1$ is a universal constant.
To evaluate the two terms on the rightmost side, denote $\kappa_l = \expect{Z_{ij}^l}$ for $l=2,4$. 
Chebyshev's inequality implies that for any $t>0$, 
$\Prob(|\Fnorm{Z_{i*}}^2 - \expect{\Fnorm{Z_{i*}}^2}| > t) \leq m \kappa_4 t^{-2}$.
Thus, a simple union bound leads to
$\Prob(\max_{1\leq i\leq k}|\Fnorm{Z_{i*}}^2 - \expect{\Fnorm{Z_{i*}}^2}| > t) \leq ks \kappa_4t^{-2}$.
Therefore,
\begin{align*}
\expect{\max_{1\leq i\leq k}|\Fnorm{Z_{i*}}^2 - \expect{\Fnorm{Z_{i*}}^2}|}
& \leq \int_{\reals_+}  \Prob(\max_{1\leq i\leq k}|\Fnorm{Z_{i*}}^2 - \expect{\Fnorm{Z_{i*}}^2}| > t) \diff t\\
& \leq k \vee s + \int_{k \vee s}^\infty ks\kappa_4 t^{-2} \diff t
= k \vee s + \kappa_4 (k \wedge s).
\end{align*}
This readily implies that 
\[
\expect{\max_{1\leq i\leq k}\Fnorm{Z_{i*}}^2} \leq (k \vee s) + \kappa_4 (k \wedge s) + \kappa_2 s \leq (1+\kappa_4 + \sqrt{\kappa_4}) (k \vee s).
\]
By symmetry, we obtain the same bound for the second term on the rightmost side of \eqref{eq:reg-oracle-upp-1}. Combining the two parts leads to the upper bound in \prettyref{eq:oracle-theta-tau}.
\label{rmk:fourth-moment}
\end{remark}

\section{Minimax rates for constrained mean matrix estimation}
\label{sec:group-sparse}

In this section, we consider two constrained mean matrix estimation problems. 
One is the submatrix sparsity constrained problem introduced in \prettyref{sec:intro-eg}, which includes the group sparsity constraint \cite{Lounici11} as a special case.
The other is the matrix completion problem \cite{Candes11,Koltchinskii11,Rohde11}, where the goal is to estimate a matrix based on noisy observations of a few entries. 
The structural constraint here is that the unknown matrix is of low rank, sometimes also referred to as rank sparsity.




\subsection{Gaussian denoising with submatrix sparsity}
Let the observed $p\times m$ matrix $Y$ be defined in \eqref{eq:sp-reg-model}.
For any matrix $X$, denote its row support and column support by 
$\rsupp(\X) = \{i: \row{\X}{i} \neq 0\}$ and $\csupp(\X) = \{j: \col{X}{j} \neq 0 \}$, respectively.
We focus on those submatrix-sparse $M$ whose row and column support have bounded cardinality.
In particular, let $k\in [p]$ and $s\in [m]$, define the following set
\begin{equation}
\label{eq:para-space}
\calF(k,s;p,m) = \{M\in \RR^{p\times m}: |\rsupp(M)| \leq k, |\csupp(M)|\leq s\}.
\end{equation}
Our goal is to determine the rate of the minimax risk
\begin{equation}
\label{eq:rate}
\Psi_{\tau}(k,s;p,m) = 
\inf_{\wt{M}}\sup_{M \in \calF(k,s;p,m)} \Expect\|\wt{M}-M\|^2_\tau
\end{equation}
for all unitarily invariant norm $\|\cdot\|_\tau$.

In the rest of this subsection, let $r = (k\wedge s) \leq (p\wedge m)$.
To state the main results, we introduce the \emph{restriction} of a symmetric gauge. 
Let $\tau$ be a symmetric gauge function on $\reals^{p\wedge m}$.
The restriction of $\tau$ on $\reals^r$, denoted by $\tau|_r$, is defined by
\begin{equation}
	\label{eq:tau-restrict}
\tau|_r(x_1,\dots, x_r) = \tau(x_1,\dots, x_r, 0,\dots,0),	
\end{equation}
for any $(x_1,\dots, x_r)\in \reals^r$.
Note that $\tau|_r$ is a symmetric gauge on $\reals^r$, whose Lipschitz constant is well-defined by \eqref{eq:Ltau}.
For notational conveniences, the $\tau$-norm of matrices of a smaller size is naturally understood per the following convention: For any $A \in \reals^{k \times s}$ with $k\in [p]$ and $s \in [m]$, the norm $\norm{A}_\tau$ is an abbreviation for $\norm{A}_{\tau|_{k \wedge s}}$, or equivalently, $\norm{A}_\tau = \norm{(\begin{smallmatrix} A & 0 \\ 0 & 0 \end{smallmatrix})}_\tau$.
In addition, we have the following property \cite{Bhatia} regarding the \ui norm of block matrices, which will be frequently used in this section:
\begin{equation}
\norm{[A~B]}_\tau \geq \norm{A}_\tau \vee \norm{B}_\tau.	
	\label{eq:tau-block}
\end{equation}

Using \eqref{eq:taunorm}, \eqref{eq:Ltau} and \eqref{eq:tau-restrict}, the following theorem paraphrases \prettyref{thm:example} and gives the minimax rates for all unitarily invariant norms.
\begin{theorem}
Let $\|\cdot\|_\tau$ be a unitarily invariant norm on $\reals^{p\times m}$.
For estimating $M$ under model \prettyref{eq:sp-reg-model} and \prettyref{eq:para-space}, the minimax rates are given by
\begin{equation}
\label{eq:rate-expression}
\Psi_{\tau}(k,s; p,m) 
\asymp (\tau|_r)^2(\ones) (k\vee s) + L_{\tau|_r}^2\, 
\pth{k\log\frac{\eexp p}{k} +  s\log\frac{\eexp m}{s}},
\end{equation}
where $r = k\wedge s$, $\ones$ is the all-one vector in $\reals^r$, $\tau|_r$ is the restriction of $\tau$ on $\reals^r$ defined in \eqref{eq:tau-restrict}, $L_{\tau|_r}$ is the Lipschitz constant of the norm $\tau|_r$ defined in \prettyref{eq:Ltau}.
\label{thm:reg-rate}
\end{theorem}

The minimax rate in \prettyref{thm:reg-rate} consists of two parts: The first term on the right side of \eqref{eq:rate-expression} is the oracle risk, which is the minimax risk if one knows the support of $M$ \emph{a priori}. See \prettyref{thm:oracle-theta-tau}.
The second term is the excess risk, which originates from the combinatorial uncertainty of the support set.

The following two examples give the specialization of \prettyref{thm:reg-rate} to the classes of Schatten norms \prettyref{eq:sq-norm} and Ky Fan norms \prettyref{eq:kf-norm}.
\begin{example}[Schatten norm]
\label{ex:reg-rate-Sch}
For the Schatten $q$-norm with $q\in [1,\infty]$,
$\tau|_r(\ones) = r^{1/q}$ and $L_{\tau|_r} = r^{(1/q-1/2)_+}$ by \eqref{eq:Lsq-Lkf}.
\prettyref{thm:reg-rate} gives the rate
\begin{equation*}
\Psi_{{\rm S}_q}(k,s; p,m) \asymp (k\wedge s)^{{2}/{q}} (k\vee s) + (k\wedge s)^{(2/q-1)_+} \pth{ k \log \frac{\eexp p}{k} + s \log \frac{\eexp m}{s}}.
\end{equation*}
Note that 
Schatten-$q$ norms satisfy
\begin{equation}
\Fnorm{\bfA} \leq 	\lsqnorm{\bfA} \leq \rank(\bfA)^{1/q-1/2}  \Fnorm{\bfA}, \quad q \in [1,2].
	\label{eq:SqF2}
\end{equation}
In view of the fact that $\Psi_{{\rm S}_q} =  (k\wedge s)^{2/q-1} \Psi_{{\rm S}_2}$, we conclude that the optimal estimator for Frobenius norm achieves the minimax rates simultaneously for all $q \in [1,2]$. It is unclear whether there exists a procedure which is simultaneously optimal for $q \in [2,\infty]$.

\end{example}

\begin{example}[Ky Fan norm]
\label{ex:reg-rate-KF}
For the Ky Fan $\ell$-norm with $\ell\in [r]$,
$\tau|_r(\ones) = \ell, L_{\tau|_r} = \sqrt{\ell}$ by \eqref{eq:Lsq-Lkf}, and so the rate is 
\begin{equation*}
\Psi_{(\ell)}(k,s; p,m) \asymp 
\ell^2(k\vee s) + \ell \pth{ k \log \frac{\eexp p}{k} + s \log \frac{\eexp m}{s}}.
\end{equation*}
\end{example}

\begin{remark}[Group sparsity]
\label{rmk:groupsparse}
When $s=m$, there is no sparsity along the columns and the problem reduces to the group sparse setting in high dimensional regression \cite{Lounici11} where each row forms a group of predictors. 
This problem has also been found useful in estimating sparse principal subspaces \cite{CMW12}. Let $\calF(k;p,m) = \{M\in \reals^{p\times m}: |\supp_r(M)|\leq k\}$.
\prettyref{thm:reg-rate} and \eqref{eq:Ltau-tau1} jointly establish the following minimax rates: 
\begin{equation}
\inf_{\wt{M}}\sup_{M\in \calF(k;p,m)} \expect \|\wt{M} - M\|_\tau^2 \asymp 
(\tau|_r)^2(\ones) (k\vee m) + L_{\tau|_r}^2\, k\log\frac{\eexp p}{k},
\label{eq:rate-group-sp}
\end{equation}
where $r = k\wedge m$ and $\ones$ is the all-one vector in $\reals^r$. The special case of \prettyref{eq:rate-group-sp} for Frobenius norm has been obtained in \cite{Lounici11}, where the lower bound matches that in \eqref{eq:rate-group-sp} and the upper bound replaces $\log\frac{\eexp p}{k}$ by $\log{p}$ but holds under more general design matrix than the orthogonal design in \prettyref{eq:sp-reg-model}.
Note that directly setting $s=m$ in \eqref{eq:rate-expression} leads to the above rate plus an extra term $L_{\tau|_r}^2 m$, while \eqref{eq:Ltau-tau1} further ensures that $L_{\tau|_r}^2 m\leq (\tau|_r)^2(\ones) (k\vee m)$.
\end{remark}

%


\subsubsection{Minimax lower bounds}
To establish the lower bound in \prettyref{thm:reg-rate}, it suffices to show that the minimax rate is lower bounded by both the oracle and the excess risk term on the right-hand side of \prettyref{eq:rate-expression} separately.
The oracle term follows straightforwardly from \prettyref{thm:oracle-theta-tau}. 
To handle the excess risk, we construct a least favorable configuration from the worst-case matrix that achieves the Lipschitz constant $L_{\tau|_r}$. The construction is \emph{probabilistic} in nature as given in the next lemma, which may be of independent interest.

\begin{lemma}
\label{lem:tau-norm-dispersion}
There exists an absolute constant $c_0 \in (0,1)$ such that the following holds: 
Let $k\geq 50$ and $s \geq 1$ be integers.
For any matrix $D\in \reals^{k\times s}$, there exists a matrix $W \in \reals^{k\times s}$ such that 
\begin{equation}
	\label{eq:fnorm-constraint}
\fnorm{W}\leq \fnorm{D}
\end{equation}
and that for any set $B\subset [k]$ with $|B| = \floor{(1-c_0) k}$, 
\begin{equation}
\|W_{B*}\|_\tau \geq c_0 \|D\|_{\tau}	
	\label{eq:balance}
\end{equation}
holds for all unitarily invariant norm $\|\cdot\|_\tau$, where $W_{B*}$ denotes the matrix formed by the rows of $W$ with indices in $B$. 
\end{lemma}

In the proof of the minimax lower bound for \prettyref{thm:reg-rate}, we use \prettyref{lem:tau-norm-dispersion} with $D$ being the maximizer which achieves $L_{\tau|_r}$ in \prettyref{eq:Lip}. 
For specific norms such as Schatten norms, we can choose a well-structured $D$ explicitly which satisfies the balanced condition in \prettyref{eq:balance} automatically. However, for general unitarily invariant norms, we need to resort to probabilistic methods to prove the existence of $W$ in \prettyref{lem:tau-norm-dispersion}, where we use a Gaussian random matrix to distribute the energy of $D$ evenly in its rows.
Since the spectra of this Gaussian random matrix scale at the same order with high probability, the unitarily invariant norms are preserved up to constants. 
It should be noted that \prettyref{lem:tau-norm-dispersion} need not hold for general norms without unitary invariance.

\begin{proof}
Recall that $r = k \wedge s$.
Since we are only interested in unitarily invariant norms, 
without loss of generality, let $D = \diag(d_1,\dots, d_r)$ with $d_1\geq \cdots\geq d_r \geq 0$.
Let $K$ be a sufficiently large fixed integer. Fix $l = \ceil{\frac{r}{2K}}$ and $j=\ceil{\frac{k}{2K}}$. 
Define $D_1 = \diag(d_1,\dots, d_l)\in \reals^{l\times s}$ and $\tilde{D} = \diag(d_1,\dots,d_l,0,\dots,0)\in \reals^{k \times s}$.
Then, for any unitarily invariant norm
\[
\|D_1\|_\tau = \|\tilde{D}\|_\tau \geq \frac{1}{2K} \|D\|_{\tau},
\]
where the last inequality is due to the triangle inequality and monotonicity of symmetric gauge functions (c.f. \prettyref{lmm:tau}).
Moreover, 
let $\tilde{U}\in \RR^{k \times k}$ have i.i.d.~$N(0,1)$ entries and let $U$ be the submatrix consisting of its first $l$ columns and all rows. 
Define the random matrix
\[
H \triangleq \tilde{U}\tilde{D} = UD_1 \in \reals^{k \times s}.
\]

Pick any $B\subset [k]$ with $|B| = k-j$.
Denote $U_{B*}$ by $U_B$ and $H_{B*}$ by $H_B$. 
Recall the Courant-Fischer minimax representation of singular values \cite[p. 75]{Bhatia}:
\[
\sigma_i(A) = \max_{\dim S = i} \min_{x \in S, \norm{x}=1} \norm{Ax}.
\]
Therefore for any matrices $M_1$ and $M_2$ and any $i\in \naturals$, 
\begin{equation}
\sigma_i(M_1 M_2)\geq \sigma_{\min}(M_1)\sigma_i(M_2),	
	\label{eq:sigma-AB}
\end{equation}
Note that $U_B \in \reals^{(k-l) \times l}$ with $l \leq k-l$. The monotonicity of symmetric gauge functions together with \prettyref{eq:sigma-AB} leads to
\[
\| H_B \|_\tau = 
\| U_B D_1 \|_\tau \geq \sigma_l(U_B) \|D_1\|_\tau.
\]
By the Davidson--Szarek inequality \cite[Theorem II.13]{Davidson01}, for any $t > 0$,
$\Prob(\sigma_l(U_B) < \sqrt{k-l} - \sqrt{l} -t ) \leq \exp(-t^2/2)$.
In addition, $j\geq l$ and $k-l\geq (K-1)j$.
Thus, for any $\beta \in (0, \sqrt{K-1}-1)$,
\begin{align*}
& \Prob\pth{\| H_B \|_\tau < (\sqrt{K-1}-1-\beta)\sqrt{j} \|D_1\|_\tau} \\
& \leq \Prob\pth{\sigma_l(U_B) < (\sqrt{K-1}-1-\beta)\sqrt{j}} \leq \exp\pth{-\frac{\beta^2 j}{2}}.
\end{align*}
Therefore, the union bound leads to
\begin{align}
& ~ \Prob\pth{\exists B\subset[k], |B|=k-j, \| H_B \|_\tau < (\sqrt{K-1}-1-\beta)\sqrt{j} \|D_1\|_\tau} \nonumber \\
\leq & ~ \sum_{B\subset[k], |B|=k-j} 
\Prob\pth{\| H_B \|_\tau < (\sqrt{K-1}-1-\beta)\sqrt{j} \|D_1\|_\tau} \nonumber \\
\leq & ~ \binom{k}{k-j} \exp\left(-\frac{\beta^2 j}{2}\right)
= \binom{k}{j} \exp\left(-\frac{\beta^2 j}{2}\right) \nonumber \\
\leq & ~ \left(\eexp \frac{k}{j} \right)^j \exp\left(-\frac{\beta^2 j}{2}\right)
 = \exp\pth{ j \pth{\log\frac{\eexp k}{j} - \frac{\beta^2}{2}} } \nonumber \\
\leq & ~ \exp \pth{\pth{\frac{k}{2K}+1} \pth{\log 2\eexp K - \frac{\beta^2}{2}}}. \label{eq:l1}
\end{align}
Moreover, \cite[Theorem II.13]{Davidson01} also implies 
\begin{align}
\Prob\pth{\sigma_1(U)\geq 2\sqrt{k}} \leq 
\Prob\pth{\sigma_1(U)\geq \sqrt{k} + \sqrt{l} + \sqrt{k/2}} 
\leq \eexp^{-k/4}.
\label{eq:l2} 
\end{align}
For sufficiently large $K \geq 25$ and $k \geq 2K$ and $\beta = \sqrt{(K-1)/2}$, the sum of the rightmost hand sides of \prettyref{eq:l1} and \prettyref{eq:l2} is less than $1$.
By the union bound, \prettyref{eq:l1} and \prettyref{eq:l2} thus imply that there exists a particular $U^*\in \reals^{k\times l}$, such that the deterministic $k\times s$ matrix $H^* = U^* D_1$ satisfies the following: 
a) $\sigma_1(U^*)\leq 2\sqrt{k}$; 
b) For all $B\subset [k], |B| = k- j$ and any unitarily invariant norm $\|\cdot\|_\tau$,
\[
\|H^*_{B*}\|_\tau \geq c \sqrt{k} \|D_1\|_\tau \geq c K^{-1} \sqrt{k} \|D\|_\tau,
\]
where $c = (\sqrt{K-1}-\sqrt{(K-1)/2}-1)/(2K)$. 
Moreover, $\Fnorm{H^*}\leq \sigma_1(U^*)\fnorm{\tilde{D}}\leq 2\sqrt{k}\fnorm{D}$, where the first inequality is due to $\fnorm{AB} \leq \fnorm{A} \opnorm{B}$.
We complete the proof by setting $c_0 = \frac{c \wedge 1}{2K}$ and $W = \frac{H^*}{2\sqrt{k}}$.
\end{proof}

Next we prove a lower bound on the packing number of matrices with submatrix sparsity with respect to the \ui $\|\cdot\|_\tau$-norm. Instead of using the abstract volume method introduced in \prettyref{sec:oracle-theta}, we give an explicit construction based on \prettyref{lem:tau-norm-dispersion} and the coding-theoretic Gilbert-Varshamov bound for packing in the Hamming space.
\begin{lemma}
There exist absolute positive constants $c_1$ and $c_2$, such that for all \ui norm $\|\cdot\|_\tau$ and all $k\in [p]$, $s\in [m]$,
	\begin{equation}
	\log \calM(B_2 \cap \calF(k,s;p,m), \|\cdot\|_{\tau}, c_1 L_{\tau|_r}) \geq c_2 \pth{k \log \frac{\eexp p}{k} + s \log \frac{\eexp m}{s}},
	\label{eq:packing.groupsparse}
\end{equation}	
where $B_2$ denotes the unit Frobenius ball and $\calM$ denotes the packing number defined in \prettyref{prop:fano}.
	\label{lmm:packing.groupsparse}
\end{lemma}
\begin{proof} 
Recall the definition of restricted gauge $\tau|_r$  and the Lipschitz constant $L_{\tau|_r}$ in \prettyref{eq:tau-restrict} and \prettyref{eq:Ltau} respectively.
By the compactness of $\{A\in \reals^{k\times s}: \fnorm{A} \leq 1\}$ and the continuity of $A \mapsto \|A\|_\tau$, there exists an $A\in \reals^{k \times s}$ such that $\fnorm{A}=1$ and $\|A\|_{\tau} = L_{\tau|_r}$. 

$1^\circ$ Assume that $k\geq 50$.
By \prettyref{lem:tau-norm-dispersion}, there exists a $k \times s$ matrix $W$ with $\fnorm{W}\leq 1$, and an absolute constant $c_0\in (0,1)$ such that for any $B\subset [k]$ with $|B| = k-k_0$, $\|W_{B*}\|_\tau  \geq c_0 L_\tau$, where $k_0 = \floor{c_0 k}$.
Now let $\calT = \{T_1,T_2,\dots, T_N\}$ be a maximal set consisting of subsets of $[p]$ with cardinality $k$, and for any $T_i\neq T_j\in \calT$, 
$|T_i \cap T_j|\leq k_0$.
By \cite[Lemma A.3]{RT11} (for $k_0 \leq \frac{p}{8}$) and \cite[Lemma 2.9]{Tsybakov09} (for $k_0 > \frac{p}{8}$)
there exist a constant $c_3$ depending only on $c_0$, such that 
\[
\log N \geq c_3 k\log{\eexp \frac{p}{k}}.
\]
Next we show that $\log \calM \geq \log N$ by constructing a packing set indexed by $\calT$.
For each $T_i\in \calT$, define $W^{(i)}\in B_2\cap \calF(k,s;p,m)$ by setting $W^{(i)}_{lj}=W_{lj} \indc{l \in T_i}\indc{j\in[m]}$. 
In other words, $W^{(i)}$ contains the rows of $W$ indexed by $T_i$ as a submatrix and the rest of the entries are zeros.
Moreover, for any $i\neq j$, $|T_i\cap T_j|\leq k_0$.
So, there exists a set $B_{ij}\subset [k]$, with $|B_{ij}|\geq k-k_0$, such that
\[
\| W^{(i)} - W^{(j)} \|_\tau \geq \|W_{B_{ij} *}\|_\tau \geq c_0 L_\tau.
\]
where the first inequality follows from \prettyref{eq:tau-block}.

$2^\circ$ Assume that $k < 50$. Let $\{e_i\}$ denote the standard basis of $\reals^p$. Note that by definition, $L_{\tau|_1}=\tau(e_1)$. Moreover, by triangle inequality, 
\[
L_{\tau|_r}=\sup_{\|x\|_2=1} \tau(x_1 e_1 + \cdots x_r e_r) \leq \sqrt{r} \tau(e_1) \leq \sqrt{50} \tau(e_1).
\]
Consider the collection of matrices $\{\ntok{V_1}{V_p}\} \subset B_2 \cap \calF(k,s;p,m)$, where $V_i=[e_i, 0, \ldots, 0]$.
Then $\|V_i-V_j\|_\tau \geq \tau(e_1)$ for any $i\neq j$.

Combining the two cases, we obtain $\log \calM \geq c_2 k \log \frac{\eexp p}{k}$ by letting $c_1 = c_0/3 \wedge \frac{1}{\sqrt{50}}$ and $c_2=c_3 \wedge \frac{1}{50}$. Exchanging the roles of row and column and replacing $(p,k)$ by $(m,s)$, we obtain that $\log \calM \geq c_2 s \log \frac{\eexp m}{s}$, completing the proof of \prettyref{eq:packing.groupsparse}.
\end{proof}

Equipped with \prettyref{lmm:packing.groupsparse}, we are ready to complete the proof of the lower bound in \prettyref{thm:reg-rate}.

\begin{proof}
By fixing the support of the submatrix to be $[k] \times [s]$, we reduce the problem to the oracle case studied in \prettyref{sec:oracle-theta} and obtain the lower bound $\Psi_{\tau}(k,s; p,m) \gtrsim (\tau|_r)^2(\ones) (k\vee s)$ by applying \prettyref{thm:oracle-theta-tau}. To prove the second term in \prettyref{eq:rate-expression}, we invoke the lower bound \prettyref{eq:fano} in \prettyref{prop:fano}, with $\epsilon = \sqrt{\frac{c_2}{4n} (k\log \frac{\eexp p}{k} + s\log \frac{\eexp m}{s})}$ and $T = B_2(\epsilon) \cap \calF(k,s;p,m)$. Then the KL diameter of $T$ satisfies $\dKL(T) \leq \dKL(B_2(\epsilon)) = 2 n \epsilon^2$. 
In view of \prettyref{lmm:packing.groupsparse} and the fact that $\calM(T, \|\cdot\|_\tau, \delta)=\calM(\alpha T, \|\cdot\|_\tau, \alpha \delta)$ for any $\alpha,\delta > 0$ and any set $T$, we have $\log \calM(T, \|\cdot\|_\tau, c_1 \epsilon L_{\tau|_r}) \geq c_2 (k\log \frac{\eexp p}{k} + s\log \frac{\eexp m}{s}) \geq C (\log 2 + \dKL(T))$ for some $C>1$. 
Here, the last inequality holds when $k\log\frac{\eexp p}{k}+s\log\frac{\eexp m}{s} \geq \lceil (2\log{2})/c_2 \rceil$.
This gives the lower bound $\Psi_{\tau}(k,s; p,m) \gtrsim L_{\tau|_r}^2\, \pth{k\log\frac{\eexp p}{k} +  s\log\frac{\eexp m}{s}}$.
If $k\log\frac{\eexp p}{k}+s\log\frac{\eexp m}{s} <\lceil (2\log{2})/c_2 \rceil$, then $L_{\tau|_r}^2\, \pth{k\log\frac{\eexp p}{k} +  s\log\frac{\eexp m}{s}}\lesssim (\tau|_r)^2(\ones)(k\vee s)\lesssim \Psi_{\tau}(k,s; p,m)$. Here, the first inequality is due to \eqref{eq:Ltau-tau1} and the second due to \prettyref{thm:oracle-theta-tau}. This completes the proof.
\end{proof}


\subsubsection{Minimax upper bounds}

In this part, 
we first define an estimator for $\wh{M}$ for $M$ and
then show the rate in \eqref{eq:rate-expression} can be achieved by this estimator.

Let the observed matrix $\bfY$ follow \eqref{eq:sp-reg-model}, and $k,s$ and the matrix norm $\|\cdot\|_\tau$ be given.
For convenience, let
\begin{equation}
\label{eq:set-IJ}
I = \rsupp(M),\qquad J = \csupp(M)
\end{equation}
be the row and column supports of $M$.
Our estimation procedure aims to select $k$ rows and $s$ columns of $\bfY$ such that any remaining block cannot be distinguished from a Gaussian noise matrix.

\paragraph{Estimation procedure}
For any $i\in [k]$ and $j\in [s]$ and any $\gamma > 0$, define
\begin{equation}
\label{eq:gaussian-sq-deviation}
\psi_\tau(i,j,p,m,\gamma) = c_1 \tau|_r(\ones) \sqrt{i\vee j} + \sqrt{\gamma}\, L_{\tau|_r} \sqrt{i\log\frac{\eexp p}{i} + j\log\frac{\eexp m}{j}}.
\end{equation}
Here, any constant $c_1\geq \sqrt{6+2\sqrt{2\pi}}$ and $\gamma \geq 4$ suffices for the upper bound argument.
Define the following collection of Cartesian product of row and column index sets
\begin{equation}
\label{eq:support-set}
\begin{aligned}
\bbB_{ks} = \bbB_{ks}(\gamma) \triangleq 
\Big\{& A\times B :  A \subset [p], B\subset [m], |A| = k, |B| = s, ~~ \mbox{and} \\
& \quad \|Y_{FG}\|_{\tau} \leq \psi_\tau(|F|,|G|, p,m,\gamma),
\,\, \\
& \quad \forall F\times G \subset \comp{(A\times B)}, |F|\leq k, |G|\leq s
\Big\}.
\end{aligned}
\end{equation}
If $\bbB_{ks}$ is not empty, we let $\hI\times \hJ$ be any Cartesian set in $\bbB_{ks}$.
Otherwise, we let $\hI = \emptyset$ and $\hJ = \emptyset$.
Our estimator is then 
\begin{equation}
\label{eq:hTT}
\wh{M} = (\widehat{M}_{ij}), \qquad \widehat{M}_{ij} = Y_{ij}\indc{i\in \hI}\indc{j\in \hJ}, \quad  i\in [p], j\in [m].
\end{equation}
If $\hI = \emptyset$ and $\hJ = \emptyset$, then $\wh{M}  = 0$.

The intuition for constructing the estimator \prettyref{eq:hTT} is the following: We know that given the support, the rate-optimal estimator is the direct observation as shown by the oracle minimax result in \prettyref{sec:oracle-theta}. The idea of the subset selector \prettyref{eq:support-set} is to choose the support sets such that the matrix outside of the support cannot be tested apart from pure Gaussian noise. 
A related idea has been used in the minimax detection of a submatrix from Gaussian additive noise in \cite{BI12}.

Now we show that $\wh{M}$ in \eqref{eq:hTT} attains the upper bound in \prettyref{thm:reg-rate}. 
Note that $\wh{M}$ requires knowledge of $k$ and $m$. Conventional penalization techniques can be used to modify $\wh{M}$ in order to achieve adaptation to the unknown row and column sparsity.
We need the following lemma regarding the unitarily invariant norm of Gaussian matrices, whose proof is deferred to \prettyref{app:pf-lemmas}.
\begin{lemma}
	\label{lem:tau-norm-noise}
Suppose $n,m\in \naturals$ and $\bfZ\in \RR^{n\times m}$ have \iid $N(0,1)$ entries.
Let $\|\cdot\|_\tau$ be a unitarily invariant norm on $\reals^{n\times m}$ where $\tau$ is a symmetric gauge on $\reals^{n\wedge m}$.
Then 
\begin{enumerate}
\item For $b=1,2,4$, there exists a universal constant $C$ such that 
\[
\Expect{\|\bfZ\|_\tau^b} \leq C (n\vee m)^{b/2}\tau^b(\ones),
\]
where $\ones$ is the all-one vector on $\reals^{n\wedge m}$.
\item For any $t>0$, 
$\Prob(\norm{\bfZ}_\tau \geq \Expect{\norm{\bfZ}_\tau} + L_\tau t) \leq \eexp^{-t^2/2}$, where $L_\tau$ is the Lipschitz constant of $\tau$ defined in \eqref{eq:Ltau}.
\end{enumerate}
\end{lemma}

\begin{proof}[Proof of \prettyref{thm:reg-rate} (Upper bound)]
	
\newcommand{\JM}{J_{\rm S}}
\newcommand{\JC}{J_{\rm C}}
\newcommand{\JO}{\hJ_{\rm O}}
\newcommand{\IM}{I_{\rm S}}
\newcommand{\IC}{I_{\rm C}}
\newcommand{\IO}{\hI_{\rm O}}

When $\bbB_{ks}\neq \emptyset$, define the following sets of row and column indices:
\begin{equation}
\begin{aligned}
\IM = I\backslash \hI, \quad & \IC = I\cap \hI,\quad && \IO = \hI\backslash I,\\
\JM = J\backslash \hJ, \quad & \JC = J\cap \hJ,\quad && \JO = \hJ\backslash J.
\end{aligned}
\end{equation}
So, $\IM$ indexes the rows in $I$ which are not included in $\hI$; $\IC$ includes the rows in $I$ which are identified by $\hI$; $\IO$ contains the rows which are over-selected by $\hI$ but not in $I$. The meaning of $\JM, \JC$ and $\JO$ are understood analogously.

Given the above definition, when $\bbB_{ks}\neq \emptyset$, the triangle inequality leads to
\begin{align}
	\label{eq:err-decomp}
\|\wh{M} - M\|_\tau & \leq \|M_{I \JM}\|_\tau + \|M_{\IM \JC}\|_\tau + \|Z_{\hI \hJ}\|_\tau.
\end{align}
We now bound each term on the right side separately.

To bound $\|M_{I \JM}\|_\tau$, the triangle inequality implies 
$\|M_{I \JM}\|_\tau \leq \|Y_{I \JM}\|_\tau + \|Z_{I \JM}\|_\tau
\leq \|Y_{I \JM}\|_\tau + \|Z_{I J}\|_\tau$,
where the second inequality comes from \prettyref{eq:tau-block}.
Moreover, since $|I|\leq k$, $|\JM|\leq |J|\leq s$, and $I\times \JM \subset \comp{(\hI\times \hJ)}$, in view of \eqref{eq:gaussian-sq-deviation}--\eqref{eq:support-set}, we have
$\|Y_{I\JM}\| \leq \psi_\tau(|I|,|\JM|,p,m,\gamma) \leq \psi_\tau(k,s,p,m,\gamma)$.
Therefore
\begin{equation}
	\label{eq:rate-miss}
\|M_{I \JM}\|_\tau \leq \psi_\tau(k,s,p,m,\gamma) + \|Z_{IJ}\|_\tau.
\end{equation}
Similar argument shows that $\|M_{\IM \JC}\|_\tau$ also satisfies the above inequality.

To control $\|Z_{\hI \hJ}\|_\tau$, we first note that 
\[
\|Z_{\hI \hJ}\|_\tau \leq \max_{\substack{F\subset [p], |F| = k\\ G\subset [m], |G| = s}} \|Z_{FG} \|_\tau.
\]
Let $\phi_\tau(k,s) = 2\psi_\tau(k,s,p,m,1)$. 
By \eqref{eq:gaussian-sq-deviation}, for any $a \geq 1$, $a\phi_\tau(k,s)\geq \psi_\tau(k,s,p,m, 4a^2)$.
Thus, we have for any $a\geq 1$
\begin{align*}
& \Prob\Big(\max_{\substack{F\subset [p], |F| = k}{G\subset [m], |G| = s}} \|Z_{FG} \|_\tau > a\phi_\tau(k,s) \Big)\\
& \leq \sum_{\substack{F\subset [p]}{|F| = k}}\sum_{\substack{G\subset [m]}{|G| = s}}
\Prob\pth{ \|Z_{FG} \|_\tau > \psi_\tau(k,s,p,m,4a^2) }\\
& \leq \binom{p}{k}\binom{m}{s} \exp\sth{ -2a^2 \pth{k\log\frac{\eexp p}{k} + s\log\frac{\eexp m}{s}  } }\\
& \leq \pth{\frac{\eexp p}{k}}^{k(1-2a^2)}\pth{\frac{\eexp m}{s}}^{s(1-2a^2)}\\
& \leq (\eexp^2 pm)^{1-2a^2}.
\end{align*}
Here, 
the second inequality is due to the Davidson-Szarek bound \cite[Theorem II.13]{Davidson01} and the fact that $\Expect\|Z_{FG}\|_\tau\leq \tau|_r(\ones)\Expect\opnorm{Z_{FG}} \leq c_1\tau|_r(\ones)\sqrt{k\vee s}$ when $c_1\geq \sqrt{6+2\sqrt{2\pi}}$, which in turn is due to \eqref{eq:Z-opnorm2} and Jensen's inequality.
The second last inequality holds because $\binom{p}{k} \leq \pth{\frac{\eexp p}{k}}^k$ for any $p\in \naturals$ and $k\in [p]$, while the last inequality is due to the fact that $k \mapsto k\log\frac{\eexp p}{k}$ is increasing for $k\in [p]$.
Thus, the last two displays, together with \prettyref{lmm:power-rate}, lead to
\begin{equation}
	\label{eq:Z-select-norm}
\Expect \|Z_{\hI \hJ}\|_\tau^2 \indc{\bbB_{ks}\neq \emptyset} \leq C \phi^2_\tau(k,s) \leq C \Psi_\tau(k,s; p,m).
\end{equation}
By \eqref{eq:err-decomp}, 
\begin{align}
& \Expect \|\wh{M} - M\|_\tau^2 \indc{\bbB_{ks}\neq \emptyset} \nonumber \\
& \leq C \pth{ \Expect \|M_{I \JM}\|_\tau^2 \indc{\bbB_{ks}\neq \emptyset} + \Expect\|M_{\IM \JC}\|_\tau^2 \indc{\bbB_{ks}\neq \emptyset} + \Expect \|Z_{\hI \hJ}\|_\tau \indc{\bbB_{ks}\neq \emptyset} } \nonumber\\
& \leq C \Expect\|Z_{IJ}\|_\tau^2 + C \psi_\tau^2 (k,s,p,m,\gamma) + C \Psi(k,s; p,m) \label{eq:rate-uppbd-1} \\
& \leq C \Psi(k,s; p,m), \label{eq:rate-uppbd-2}
\end{align}
where \eqref{eq:rate-uppbd-1} is due to \eqref{eq:rate-miss} and \eqref{eq:Z-select-norm}, and \eqref{eq:rate-uppbd-2} comes from \prettyref{lem:tau-norm-noise} and the fact that for any fixed $\gamma$, $\psi_\tau^2(k,s,p,m,\gamma) \asymp \Psi(k,s; p,m)$.

To complete the proof, we only need to bound
\[
\Expect \|\wh{M} - M\|_\tau^2 \indc{\bbB_{ks} = \emptyset} = 
\Expect \|M\|_\tau^2 \indc{\bbB_{ks} = \emptyset}.
\]
Note that $\|M_{IJ}\|_\tau \leq \|Y_{IJ}\|_\tau + \|Z_{IJ}\|_\tau$. 
When $\bbB_{ks} = \emptyset$, by \eqref{eq:support-set},
$\|Y_{IJ}\|_\tau  \leq \psi_\tau(k,s,p,m,\gamma)$.
Therefore conditioned on the event $\{\bbB_{ks} = \emptyset\}$, the triangle inequality leads to
\[
\|M\|_\tau = \|M_{IJ}\|_\tau \leq  \psi_\tau(k,s,p,m,\gamma) + \|Z_{IJ}\|_\tau.
\]
Thus, 
$\Expect \|M\|_\tau^2 \indc{\bbB_{ks} = \emptyset} \leq C \pth{ \psi_\tau^2 (k,s,p,m,\gamma) + \Expect \|Z_{IJ}\|_\tau^2 }\leq C\Psi_\tau(k,s; p,m)$.
This completes the proof.
\end{proof}

\subsection{Matrix completion}
Let $M$ be a $k\times s$ matrix of interest. Let $\{X_1,\dots,X_n\}$ be \iid uniform on $\calX = \{e_j(k) e_l'(s), j\in [k], l\in [s]\}$, where $\{e_j(k),j\in [k]\}$ are the standard bases in $\reals^k$.
Our goal is to estimate $M$ based on the observations
\begin{equation}
\label{eq:matrix-completion}
Y_i = \Tr(X_i' M) + \sigma Z_i,\qquad i\in [n],
\end{equation}
where $\sigma>0$ is the noise level and $Z_i$ are \iid~$N(0,1)$ and independent of $\{X_1,\dots,X_n\}$.
The interesting case is when the number of observations, $n$, is much smaller than the number of entries, $ks$.
To make the problem feasible, we assume that $M$ has low rank and bounded entries, \ie, $M$ belongs to the set
\[
\calM(r,a) = \{M = (M_{ij})\in \reals^{k\times s}: \rank(M)\leq r,\, \|M\|_{\ell_\infty} \leq a\},
\]
where $\|M\|_{\ell_\infty} = \max_{i,j}|M_{ij}|$.


To establish a general lower bound for any \ui norm, we need the following lemma (proved in \prettyref{app:pf-lemmas}) to control the KL divergence between distributions of the observed $Y_i$'s based on different underlying mean matrices. 
\begin{lemma}
Let $M \in \reals^{k \times s}$. Denote by $P_{M}$ the joint distribution of $\{(Y_i, X_i):i\in [n]\}$ defined in \prettyref{eq:matrix-completion}.
Then 
\begin{align*}
D(P_{M_1} \, || \, P_{M_2}) 
& \leq \frac{1}{2 \sigma^2} \pth{1 - \pth{1-\frac{1}{ks}}^n}   \fnorm{M_1-M_2}^2 \\
& \leq \frac{1}{2 \sigma^2} \frac{n}{ks} \fnorm{M_1-M_2}^2.
\end{align*}
\label{lmm:erasure}
\end{lemma}

Using \prettyref{lmm:erasure} and the volume approach, we obtain the following result on the minimax lower bounds for matrix completion.
\begin{theorem}
\label{thm:matrix-completion}
Let $\|\cdot\|_\tau$ be any \ui norm. Let $1\leq n\leq ks$. The minimax risk for estimating $M$ under model \eqref{eq:matrix-completion} satisfies
\[
\inf_{\wt{M}} \sup_{M\in \calM(r,a)} \Expect \|\wt{M} - M\|_\tau^2 
\gtrsim (\sigma\wedge a)^2\frac{ks}{n} (k \vee s) (\tau|_r)^2(\ones),
\]
where $\ones$ is a vector of all ones in $\reals^r$.
Moreover, there exists an absolute constant $c_0 \in (0,1)$, such that
\[
\inf_{\wt{M}} \sup_{M\in \calM(r,a)}\Prob\sth{\|\wt{M} - M\|_\tau^2 \geq c_0 (\sigma\wedge a)^2\frac{ks}{n} (k \vee s) (\tau|_r)^2(\ones)} \geq c_0.
\]
\end{theorem}

\begin{remark}
Consider the case where $\sigma\asymp a$.
For Schatten-$q$ norms, we have $\tau|_r(\ones) = r^{1/q}$. Hence \prettyref{thm:matrix-completion} leads to
\begin{equation}
	\label{eq:completion-schatten}
\inf_{\wt{M}} \sup_{M\in \calM(r,a)} \Prob\sth{\|\wt{M} - M\|_{{\rm S}_q}^2 
\geq c_0 \frac{\sigma^2 ks}{n}r^{2/q}(k\vee s)} \geq c_0
\end{equation}
for some absolute constant $c_0 \in (0,1)$.
Corollary 2 in \citet{Koltchinskii11} showed that for some $\epsilon > 0$, when $n > (k\wedge s)\log^{1+\epsilon}(k\vee s)$, the squared Schatten-2 loss (\ie, the squared Frobenius loss) of an estimator $\wh{M}_{\rm KLT}$ obtained via nuclear norm penalization is upper bounded by the rate in \eqref{eq:completion-schatten} times $\log(k\vee s)$ with probability at least $1-{3}/{(k+s)}$.
In view of \prettyref{eq:SqF2}, with probability at least $1-3/(k+s)$, for all $q\in [1,2]$,
\[
\sup_{M\in \calM(r,a)} \|\wh{M}_{\rm KLT} - M\|_{{\rm S}_q}^2 
\lesssim \frac{\sigma^2 ks}{n}r^{2/q} (k\vee s) \log(k\vee s).
\]
The above result shows that when $\sigma\asymp a$, the probabilistic lower bounds in \prettyref{thm:matrix-completion} are tight up to a log factor for all Schatten-$q$ norms with $q\in [1,2]$. 

In fact, the lower bounds in \prettyref{thm:matrix-completion} also apply to other sampling models. For example, instead of the ``sampling with replacement'' model in \prettyref{eq:matrix-completion}, \prettyref{lmm:erasure} and, consequently, \prettyref{thm:matrix-completion} apply verbatim to the corresponding ``sampling without replacement'' model where each basis in $\calX$ is chosen with probability $\frac{n}{ks}$.
\end{remark}

\begin{proof}[Proof of \prettyref{thm:matrix-completion}]
Without loss of generality, assume that $k \geq s$.
Restricting to those matrices where only the first $r$ columns are non-zero, it is sufficient to prove the following lower bound:
\[
\inf_{\wt{M}} \sup_{M\in B_{\ell_\infty}(a)} \Expect \|\wt{M} - M\|_\tau^2 
\gtrsim (\sigma\wedge a)^2\frac{ks}{n} k (\tau|_r)^2(\ones),
\]
where $B_{\ell_\infty}(a)$ denotes the $\ell_\infty$-ball in $\reals^{k \times r}$.
Set $T = B_{\ell_\infty}(a \wedge \sigma)$. Then $\vol(T)=(a\wedge \sigma)^{kr}$. 
Moreover, in view of \prettyref{lmm:erasure} and the fact that $n \leq kr$, the KL-diameter of $T$ satisfies $\dKL(T) \leq \frac{1}{2\sigma^2}\frac{n}{ks} (a \wedge \sigma)^2 kr \leq \frac{kr}{2}$.
Set $\epsilon = c (\sigma \wedge a) k \sqrt{\frac{s}{n}} \tau(\ones)$ for some small absolutely constant $c$. In view of \prettyref{lmm:urysohn}, we have $\vol(B_{\|\cdot\|_\tau}(\epsilon))^{\frac{1}{kr}} \lesssim \frac{\epsilon}{\sqrt{k} \tau(\ones)} \lesssim (\sigma \wedge a) \sqrt{\frac{ks}{n}} \leq \sigma \wedge a$. The lower bound of order $\epsilon^2$ then follows from an application of \prettyref{prop:fano}.
The lower bound in probability follows from \eqref{eq:fano-vol-highprob} by using the same $T$ and $\epsilon$.
\end{proof}



\section{Beyond normal mean models}
	\label{sec:beyond}
To demonstrate the applicability of the volume method beyond normal mean models, 
we switch in this section to the problems of covariance matrix estimation and Poisson rate matrix estimation with no structural constraints. 
The volume approach can be successfully employed in these problems to derive optimal minimax rates for all \ui norms.
The main difference is that, unlike in normal mean models, the KL neighborhood is a convex body induced by the model, which need not be an Euclidean ball.
Note that these unconstrained problems are non-trivial. 
To the best of our knowledge, even for estimating a covariance matrix with independent normal samples, the minimax rate under the squared Frobenius norm is not known for all sample size, dimension, and spectral radius. 
Moreover, similar to the case of normal mean models, rates in these unconstrained problems are instrumental in obtaining the rates in their constrained variants.
	
\subsection{Covariance matrix estimation}
\label{sec:oracle-sigma}

Let $X$ denote the observed $n \times k$ data matrix, whose rows $\ntok{\row{X}{1}}{\row{X}{n}}$ are independently drawn from $N(0,\Sigma)$. 
A sufficient statistic for $\Sigma$ is the sample covariance matrix $S = \frac{1}{n} X' X$. 

Without assuming additional covariance structure, we consider the following parameter space for $\Sigma$:
\begin{equation}
	\Xi(k,\lambda) = \{\Sigma \in \sfS_k^+: \opnorm{\Sigma} \leq \lambda\},
	\label{eq:para-space-sigma}
\end{equation}
which is simply the operator norm ball of radius $\lambda$ in the space of $k\times k$ symmetric semi-positive definite matrices.

We have the following analogous result to \prettyref{thm:oracle-theta-tau} for covariance matrices. 
The main difference is that instead of \prettyref{eq:KL-GLM}, the KL divergence in the covariance model is given by
\begin{equation}
\KL{N(0,\tSS)}{N(0,\Sigma)}
= \frac{1}{2} \Tr(\SS^{-1} \tSS -\bfI)- \frac{1}{2} \log \frac{\det \tSS}{\det \SS}.	
	\label{eq:KL-cov}
\end{equation}
Therefore the KL neighborhood in the covariance model is not a Frobenius ball, which requires additional volume estimates via the inverse Santal\'o inequality in the lower bound argument.
\begin{theorem}
For any $n,k \in \naturals$, any $\lambda > 0$, and any unitarily invariant norm $\|\cdot\|_\tau$, where $\tau$ is a symmetric gauge function on $\reals^{k}$,
	\begin{equation}
\inf_{\tSS}\sup_{\SS \in \Xi(k,\lambda)} \Expect{\|\tSS - \SS\|_\tau^2} \asymp \pth{\frac{k}{n} \wedge 1} \lambda^2 \tau^2(\ones).
	\label{eq:oracle-sigma-tau}
\end{equation}
	\label{thm:oracle-sigma-tau}
\end{theorem}

It is interesting to compare \prettyref{thm:oracle-sigma-tau} to the classical results focusing on the \emph{exact} minimax risk of estimating the covariance matrices in the low-dimensional regime. 
For instance, using invariance theory, Stein \cite{Stein56} proved that if $k \leq n$, any constant multiple of the sample covariance matrix is not minimax with respect to the KL loss \prettyref{eq:KL-cov}  (also known as the Stein loss). He also obtained the minimax estimator for this problem.
In contrast, our focus here is to investigate the minimax \emph{rate}, the non-asymptotic characterization of the minimax risk modulo constants. In particular, we see that the sample covariance matrix is minimax \emph{rate}-optimal for all triples $(k,n,\lambda)$ and all unitarily invariant norms. This conclusion, even in the simplest setting of quadratic loss (squared Frobenius norm), seems to be new in the literature.

Before proceeding to the proof, we discuss the implications of \prettyref{thm:oracle-sigma-tau} and how the minimax rate depends on various parameter of the problem:
\begin{enumerate}

	\item Note that	the dependence of the minimax risk in \prettyref{thm:oracle-sigma-tau} on the largest spectral norm through $\lambda^2$ is natural. The reasons are two-fold: First, since the the covariance model is a scale model, the Kullback-Leibler divergence is scaling invariant in the sense that 
	\[
	\KL{N(0,\Sigma_0)}{N(0,\Sigma_1)} = \KL{N(0,\lambda  \Sigma_0)}{N(0,\lambda \Sigma_1)}.
	\]
	  On the other hand, the loss in terms of squared norm scales quadratically with $\lambda^2$. Second, the magnitude of the ``effective noise'' matrix $S - \Sigma$ also scales with the spectral norm of $\Sigma$.
	
	\item When the dimension $k$ exceeds the sample size $n$, there is no way to estimate under any unitarily invariant norm in the sense that the minimax error is equivalent to the radius of the parameter space, which can be achieved by any fixed element of the parameter space. 
This phenomenon does not apply to the mean model, where estimating by the observation is always rate optimal. The underlying reason lies in the difference of the information geometry between the two models: The KL neighborhood in the Gaussian mean model coincides with the Frobenius ball, whereas in the covariance model, as the diameter grows, the Kullback-Leiber neighborhood evolves from a Frobenius ball into a spectral norm ball. See the proof of \prettyref{thm:oracle-sigma-tau} for more details.
	
	\item Analogous to the discussion of \prettyref{thm:oracle-theta-tau} in \prettyref{rmk:tau1}, the minimax rate in \prettyref{thm:oracle-sigma-tau} is also proportional to $\tau^2(\ones)$, 
	which suggests that the worst-case prior are in general position.
\end{enumerate}


\begin{proof}[Proof of \prettyref{thm:oracle-sigma-tau}]
We first establish the upper bound. Denote the sample covariance matrix by $S=\frac{1}{n}XX'$. Then $\tS=\Sigma^{-\frac{1}{2}} S \Sigma^{-\frac{1}{2}}$ is a $k\times k$ standard Wishart matrix with $n$ degrees of freedom. Applying the deviation inequality in \cite[Proposition 4]{CMW12}, we have $\expect{\opnorm{\tS-I_k}^2} \lesssim \frac{k}{n}+\frac{k^2}{n^2}$. Since $\opnorm{S-\SS} \leq \opnorm{\SS} \opnorm{\tS-I_k}$, we have $\expect{\opnorm{S-\SS}^2} \lesssim \lambda^2 \pth{\frac{k}{n}+\frac{k^2}{n^2}}$. Since $\norm{\cdot}_\tau \leq \tau(\ones) \opnorm{\cdot}$, we have $\expect{\norm{S-\SS}_\tau^2} \lesssim \lambda^2 \tau^2(\ones) \pth{\frac{k}{n}+\frac{k^2}{n^2}}$. On the other hand, estimating by zero gives $\norm{\SS}_\tau \leq \lambda \tau(\ones)$. The minimax upper bound in \prettyref{eq:oracle-sigma-tau} follows upon noticing that $\pth{\frac{k}{n}+\frac{k^2}{n^2}} \wedge 1 \asymp \frac{k}{n} \wedge 1$.

It remains to prove the lower bound. Let $r>0$.	Define
\begin{equation}
	K(r) \triangleq \frac{\lambda}{2} I + \frac{\lambda}{2} B_{\sym_2}(\sqrt{2r}) \cap B_{\sym_\infty}(1/2) \cap \sfS_k.
	\label{eq:Kr}
\end{equation}
Next we show that the Kullback-Leibler diameter of $\K(r)$ satisfies
\begin{equation}
\dKL(K(r)) \leq 16 r.	
	\label{eq:Kr-diam}
\end{equation}
To see this, first note that the matrices in $K(r)$ is well-conditioned: $\sigma_{1}(\Sigma) \leq \frac{3\lambda}{4}$ and $\sigma_{k}(\Sigma) \geq \frac{\lambda}{4}$  for any $\Sigma \in K(r)$. Then for any $\Sigma_0,\Sigma_1 \in K(r)$, $\sigma_1(\SS_0^{-1} \SS_1) \leq \sigma_1(\SS_0^{-1})  \, \sigma_1(\SS_1) \leq 3$ and $\sigma_k(\SS_0^{-1} \SS_1) \geq \sigma_k(\SS_0^{-1}) \, \sigma_k(\SS_1) \geq \frac{1}{3}$. Consequently, we have
\begin{align}
\KL{N(0,\Sigma_1)}{N(0,\Sigma_0)}
= & ~ \frac{1}{2} \Tr(\SS_0^{-1} \SS_1 -\bfI)- \frac{1}{2} \log \frac{\det \SS_1}{\det \SS_0}	\nonumber \\
= & ~ \frac{1}{2} \sum_{i=1}^k \sigma_i(\SS_0^{-1} \SS_1) - 1 - \log \sigma_i(\SS_0^{-1} \SS_1)	\nonumber \\
\leq & ~ \frac{1}{2} \fnorm{\SS_0^{-1} \SS_1 - I}^2	\label{eq:logtaylor} \\
\leq & ~ \frac{1}{2} \opnorm{\SS_0^{-1}}^2 \fnorm{\SS_0 -\SS_1}^2	\label{eq:logtaylor2} \\
\leq & ~ 16 r, \nonumber 
\end{align}
where \prettyref{eq:logtaylor} follows from $\log(1+x) \geq x - x^2$ for all $x \in [-\frac{2}{3},2]$ and \prettyref{eq:logtaylor2} follows from $\fnorm{AB} \leq \opnorm{A} \fnorm{B}$.


Next we use the inverse Santal\'o's inequality to lower bound the volume of $K(r)$. Let $d_k=k(k+1)/2$ denote the dimension of $\sfS_k$. Recall that $\tG \triangleq G_{\sfS_k} = \frac{G+G'}{2}$ denote the Gaussian ensemble on $\sfS_k$ (GOE($k$)). By the translation and scaling properties of the volume measure, we have
\begin{equation}
\vol(K(r))^{\frac{1}{d_k}} = \frac{\lambda}{2} \vol(B_{\sym_2}(\sqrt{2r}) \cap B_{\sym_\infty}(1/2) \cap \sfS_k)^{\frac{1}{d_k}}.
	\label{eq:volKr1}
\end{equation}
Setting $r=d_k/n$ and applying \prettyref{lmm:santalo}, we have
\begin{align}
\vol(B_{\sym_2}(\sqrt{2k^2/n}) \cap B_{\sym_\infty}(1/2) \cap \sfS_k)^{\frac{1}{d_k}}	
\geq & ~ \frac{c_0}{\expect{\frac{\fnorm{\tG}}{\sqrt{2k^2/n}}  \vee 2 \, \opnorm{\tG} }}	\nonumber \\
\geq & ~ \frac{c_0}{ \sqrt{\frac{n}{2k^2}} \expect{\fnorm{\tG}} + 2 \, \expect{\opnorm{\tG} }}	\nonumber \\
\geq & ~ \frac{c_0'}{ \sqrt{k \vee n}}	\label{eq:volKr2},
\end{align}
where \prettyref{eq:volKr2} follows from $\expect{\fnorm{\tG}} \leq (\expect{\fnorm{\tG}^2})^{1/2} = k$ and $\expect{\opnorm{\tG}} \leq \sqrt{2k}$. Here $c_0,c_0'$ are universal constants. On the other hand, by Urysohn's inequality \prettyref{eq:urysohn-E} and the fact that $\vol(B_{\|\cdot\|_\tau}(\epsilon) \cap \sfS_k)^{\frac{1}{d_k}} \asymp \frac{1}{k}$,
\begin{equation}
\vol(B_{\|\cdot\|_\tau}(\epsilon) \cap \sfS_k)^{\frac{1}{d_k}} \leq  \frac{\epsilon \tau_*(\ones) \expect{\|\tG\|_{\sym_\infty}}}{d_k} \asymp
 \frac{\epsilon}{\sqrt{k} \tau(\ones)}.
	\label{eq:vol-tau}
\end{equation}
Combining \prettyref{eq:volKr1}, \prettyref{eq:volKr2} and \prettyref{eq:vol-tau} yields
\begin{equation}
	\pth{\frac{\vol(K(r))}{\vol(B_{\|\cdot\|_\tau}(\epsilon)  \cap \sfS_k)}}^{\frac{1}{d_k}} \geq \frac{c_0' \lambda \sqrt{k} \tau(\ones)}{\sqrt{k \vee n} \epsilon}.
	\label{eq:volratio-Kr}
\end{equation}

%
%
%
%


Set $\epsilon = c \lambda \tau(\ones) \sqrt{\frac{k}{n} \wedge 1}$ for $c=\frac{c_0'}{64}$. In view of \prettyref{eq:Kr-diam} and \prettyref{eq:volratio-Kr}, applying \prettyref{prop:fano} to $T = K(r)$ yields the desired lower bound.
\end{proof}

\subsection{Poisson rate matrix estimation}
	\label{sec:poisson}

Consider the following Poisson model:
\begin{equation}
	X_{ij} \stackrel{ind.}{\sim} \text{Poisson}(\lambda_{ij}) 
	\label{eq:poisson}
\end{equation} 
where the intensity matrix $\Lambda$ belongs to the following parameter set
\begin{equation}
	\Gamma(k, s, \lambda) = \{\Lambda \in \reals_+^{k \times s}: \lambda_{ij} \leq \lambda\}.
	\label{eq:poi-parameter}
\end{equation}
The goal is to estimate the rate matrix $\LL$ based on the observation $X$. This problem is closely connected to Poisson denoising, which has applications in photon-limited medical and astronomical imaging,  and computer vision \cite{Donoho93,WN03,ZFS08}.


\begin{theorem}
For any $k,s \in \naturals$, any $\lambda > 0$, and any unitarily invariant norm $\|\cdot\|_\tau$, where $\tau$ is a symmetric gauge function on $\reals^{k \wedge s}$,
	\begin{equation}
\inf_{\hLL}\sup_{\LL \in \Gamma(k, s, \lambda)} \Expect{\| \hLL - \LL\|_\tau^2} 
\gtrsim (k \vee s) \tau^2(\ones) (\lambda \wedge \lambda^2).
	\label{eq:oracle-poisson}
\end{equation}
	\label{thm:oracle-poisson}
\end{theorem}
\begin{remark}
For squared Schatten-$q$ norm losses with $q \in [1,2]$, we have the following tight minimax rates for all $\lambda > 0$:
\begin{equation}
	\inf_{\hLL}\sup_{\LL \in \Gamma(k, s, \lambda)} \Expect{\Lsqnorm{\hLL - \LL}^2} \asymp (k \vee s) (k \wedge  s)^{2/q} \, (\lambda \wedge \lambda^2).
	\label{eq:rate-sq-poisson}
\end{equation}
To show the upper bound, first note that $\inf_{\hLL}\sup_{\LL \in \Gamma(k, s, \lambda)} \Expect{\|\hLL - \LL\|_{{\rm S}_2}^2} \lesssim k  s \, (\lambda \wedge \lambda^2)$, 
achieved by $\hLL=X$ or $\hLL=0$ when $\lambda \geq 1$ or $< 1$, respectively. Then the rate in \prettyref{eq:rate-sq-poisson} follows from \prettyref{eq:SqF2} with $r=k \wedge s$.

%
%
\end{remark}

\begin{remark}
We also remark that estimation by the observed $X$ yields
\[
\sup_{\LL\in \Gamma(k,s,\lambda)}\Expect \|X-\LL\|_\tau \leq \tau(\ones)\sqrt{(k\vee s)\lambda}\,.
\]
Note that $X-\LL$ has independent mean zero entries with $\Expect (X_{ij}-\lambda_{ij})^2 = \lambda_{ij}$ and $\Expect (X_{ij}-\lambda_{ij})^4 = \lambda_{ij} + 3\lambda_{ij}^2$.
The last display thus holds due to \cite[Theorem 2]{Latala05} and the fact that $\|X-\LL\|_\tau \leq \opnorm{X-\LL}\tau(\ones)$.
\end{remark}

\begin{proof}[Proof of \prettyref{thm:oracle-poisson}]
We now turn to the proof of the lower bound.
Consider the following subset of the parameter space
\[
K = \LL_0 + B_{\ell_\infty}(\lambda/4) \cap B_2(\sqrt{\lambda ks}),
\]
where $\LL_0$ is all-zero matrix except the top-left element being $\frac{3}{4}\lambda$. Then it is straightforward to verify that 
\[
K \subset B_{\ell_\infty}(\lambda) \backslash B_{\ell_\infty}(\lambda/2) \subset \Gamma(k, s, \lambda).
\]
In order to apply \prettyref{prop:fano}, we bound the volume and the KL-diameter of $K$ from below and above, respectively. Note that $B_{\ell_\infty}(a) \subset  B_2(k)$ for all $a \leq 1$. Therefore
\begin{align*}
	\vol(K)^{\frac{1}{ks}}
= & ~ \sqrt{\lambda} \, \vol(B_{\ell_\infty}(\sqrt{\lambda}/4) \cap B_2(\sqrt{ks}))^{\frac{1}{ks}}	\\
\geq & ~ \sqrt{\lambda} \, \vol\Big(B_{\ell_\infty}\Big(\frac{\sqrt{\lambda}}{4} \wedge 1\Big)\Big)^{\frac{1}{ks}} \\
= & ~ \frac{1}{4} \sqrt{\lambda \wedge \lambda^2}. 
\end{align*}
	Note that the KL divergence in the Poisson model is given by
\[
D(\Poisson(\lambda_1) \, || \, \Poisson(\lambda_0)) = \lambda_1 \log \frac{\lambda_1}{\lambda_0}  - \lambda_1 + \lambda_0 \leq \frac{(\lambda_1-\lambda_0)^2}{\lambda_0},
\]
where the last inequality is due to $\log(1+t) \leq t$ for all $t > -1$. 
Therefore, we conclude that for any $\LL,\tilde{\LL} \in K$, $D(P_{X|\Lambda} \, || \, P_{X|\tilde{\Lambda}}) = \sum_{i,j=1}^k \lambda_{ij} \log \frac{\lambda_{ij}}{\tilde{\lambda}_{ij}}  - \lambda_{ij} + \tilde{\lambda}_{ij} \leq \frac{2 \fnorm{\Lambda -\tilde{\Lambda}}^2}{\lambda}$.
Therefore the KL-diameter of $K$ satisfies $\dKL(K) \leq 2 ks$.

Set $\epsilon = c \tau(\ones) \sqrt{(k \vee s) (\lambda \wedge \lambda^2)}$ for some small absolutely constant $c$. In view of the Urysohn's lemma (see \prettyref{lmm:urysohn} and also \prettyref{eq:vol-tau}), we have $\vol(B_{\|\cdot\|_\tau}(\epsilon))^{\frac{1}{ks}} \lesssim \frac{\epsilon}{\sqrt{k \vee s} \tau(\ones)} \lesssim \sqrt{\lambda \wedge \lambda^2}$.
The lower bound of order $\epsilon^2$ then follows from an application of \prettyref{prop:fano} to $T=K$. 
\end{proof}


\section{Discussion}
	\label{sec:discuss}


In this paper, we developed a novel unified approach to study non-asymptotic minimax estimation of large matrices with respect to all squared unitarily invariant norm losses in a variety of settings. 
In addition to the settings considered in the current paper, the machinery is potentially also useful for determining the minimax rates of other large matrix estimation problems.

For the ease of exposition, we have focused on those loss functions which are the square of certain norms. The squaring operation is certainly non-essential, since, in view of the high-probability bound in \prettyref{rmk:highprob}, our lower bound technique is applicable to any loss of the form $\ell(M,M') = w(\norm{M-M'}_\tau)$ for some increasing function $w: \reals_+ \to \reals_+$.
 On the other hand, the tightness of our results hinges on the unitary invariance of the loss functions. Minimax rates for norms lacking unitary invariance, \eg, vector induced norms considered in  \citet{CLZ13}, are outside the scope of the present paper.

Due to the generality of the loss functions considered, the primary focus of the paper is on determining the minimax rates. 
There are two related but different questions that pose challenging future research problems. 
\begin{enumerate}
	\item \emph{Computational complexity.} For a given model and a given norm loss, does there exist an estimator which is both minimax rate-optimal and computationally efficient?
	For many models, the answer to this question seems to be highly dependent on the loss function. For instance, for estimation under row-wise (group) sparsity \cite{Lounici11} which is a special case of the submatrix sparsity model studied in \prettyref{sec:group-sparse}, the minimax rate under squared Frobenius loss can be obtained via row-norm thresholding \cite[Section 4]{CMW12}. On the other hand, we are not aware of a procedure that attains the operator-norm minimax rates.
	
	
	\item \emph{Loss adaptivity.} When can a single estimator attain the (near) optimal rates with respect to a collection of norm losses?
	The results obtained in the current paper give examples on the affirmative side. For example, the estimator by \cite{Koltchinskii11} is simultaneously near-optimal for matrix completion with respect to all Schatten-$q$ norm losses with $q\in [1,2]$. Likewise, as shown in \prettyref{ex:reg-rate-Sch}, for the submatrix sparsity problem, the optimal estimator for Frobenius norm is simultaneously optimal for all Schatten-$q$ norm losses with $q\in [1,2]$.
	A better understanding of this phenomenon depends crucially on first understanding the minimax rates under different norm losses, for which our machinery can be instrumental.
	Answers to this question can also help researchers tackle the previous question on computationally efficient estimators, and this time with the extra delight of hitting multiple birds with one stone.
\end{enumerate}

\appendix
\section{Technical details}
\label{app:pf-lemmas}

First we state a lemma used in the proof of \prettyref{thm:oracle-theta-tau}.
\begin{lemma}
\label{lmm:power-rate}
For a random variable $X$, suppose that for a constant $\phi$ and a function $h$ such that for all $a \geq 1$, 
$\Prob(|X|>a\phi)\leq h(a)$.
Then for all $b>0$ such that $\int_1^\infty a^{b-1}h(a)\diff a \leq C <\infty$, there exists a constant $C'$ that depends only on $b$ such that
$\Expect |X|^b \leq C' \phi^b$.
\end{lemma}
\begin{proof}
Note that
\begin{align*}
\Expect |X|^b
= & ~ b\int_0^\infty t^{b-1} \, \Prob(|X| \geq t) \diff t \\
\leq & ~ \phi^b + b\int_{\phi}^\infty t^{b-1} \, \Prob(|X| \geq t) \diff t 
 = \phi^b + b\phi^b \int_{1}^\infty a^{b-1} \Prob(|X| \geq a \phi) \diff a \\
\leq & ~ \phi^b \pth{1 +  b\int_{1}^\infty a^{b-1} f(a) \diff a}.
\end{align*}
This completes the proof.
\end{proof}

Next we provide proofs for various technical lemmas used in the paper.
\begin{proof}[Proof of \prettyref{lmm:tau}]
The definition of symmetric gauge function implies that $\tau$ is an absolute norm on $\reals^d$ \cite[p.438]{horn}, which in turn implies the desired monotonicity \cite[Theorem 5.5.10]{horn}.
	
	The fact that $\tau_*$ is a symmetric gauge can be found in \cite[Exercise IV.1.13]{Bhatia}. By definition of the dual norm,
\begin{equation}
\tau_*(\ones) = \sup_{\tau(y) \leq 1} \iprod{\ones}{y} = \sup_{\tau(y) \leq 1, y \in \reals^d_+} \iprod{\ones}{y}.	
	\label{eq:tauonedual}
\end{equation}
Let $\pi$ denote the random permutation matrix on $[d]$. For any $y \in \reals^d_+$, $\expect{\pi(y)} = \frac{\iprod{\ones}{y}}{d} \ones$, which satisfies $\iprod{\ones}{\expect{\pi(y)}} = \iprod{\ones}{y}$. By the convexity of $\tau$ and Jensen's inequality, $\tau(\expect{\pi(y)}) \leq \expect{\tau(\pi(y))}=\tau(y) \leq 1$. Therefore the supremum in \prettyref{eq:tauonedual} is achieved at the constant vector with unit $\tau$ norm, \ie, $\frac{1}{\tau(\ones)} \ones$.
\end{proof}

\begin{proof}[Proof of \prettyref{lem:tau-norm-noise}]
By \prettyref{lmm:tau},
\[
\norm{\bfZ}_{\tau}^b \leq \sigma_1(\bfZ)^b \tau^b(\ones).
\]
\citet[Theorem II.13]{Davidson01} shows that $\Prob(\sigma_1(\bfZ)>\sqrt{n}+\sqrt{m}+t)\leq \eexp^{-t^2/2}$, and so
$\Prob(\sigma_1(\bfZ) > a \cdot 2\sqrt{n}) \leq \eexp^{-2n(a-1)^2}$.
Since $\int_1^\infty a^{b-1}\eexp^{-2n(a-1)^2}\diff a<\infty$,
\prettyref{lmm:power-rate} implies $\Expect{\sigma_1(Z)^b}\leq C(n\vee m)^{b/2}$, which, together with the second last display, leads to the first claim.


Turn to the second claim.
Following the discussion in \prettyref{sec:norm}, we have $f(\bfZ) = \norm{\bfZ}_\tau$ is a Lipschitz function on $\RR^{nm}$ with Lipschitz constant $L_\tau$.
The second claim then follows directly from the concentration of measure in Gaussian space \cite{MS86}.
\end{proof}

\begin{proof}[Proof of \prettyref{lmm:erasure}]
	It is sufficient to consider the following vector problem: Let $X \sim N(\theta,I_d)$. Let $\ntok{i_1}{i_n}$ be \iid~uniform on $[d]$. The observed data are $Y = (i_j, X_{i_j})_{j \in [n]}$, whose distribution is denoted by $P_Y^\theta$. 
	We prove that 
	\[
	D(P_Y^{\theta_1} \, || \, P_Y^{\theta_2}) \leq \frac{1}{2 \sigma^2} \pth{1 - \pth{1-\frac{1}{d}}^n}   \fnorm{\theta_1 - \theta_2}^2.
	\]
	which yields the desired lower bound upon setting $d = ks$. To this end, denote the \emph{set} (not multiset) $I = \{i_j: j \in [n] \}$. 
	For probability transition kernels $P_{Y|X}$ and $Q_{Y|X}$ and some prior $\pi$ for $X$, denote the respective marginals of $Y$ by $P_Y$ and $Q_Y$. We use the standard information-theoretic notation for conditional KL divergence $D(P_{Y|X} \, || \, Q_{Y|X} | \pi) = \Expect_{X \sim \pi} D(P_{Y|X} \, || \, Q_{Y|X})$. Then by the convexity of $(P,Q) \mapsto D(P \, || \, Q)$, we have $D(P_{Y|X} \, || \, Q_{Y|X} | \pi) \leq D(P_{Y|X} \, || \, Q_{Y|X} | \pi)$. 
	Therefore
	\begin{align*}
	D(P^{\theta_1}_{Y} \, || \, P^{\theta_2}_{Y})
	\leq & ~ D(P^{\theta_1}_{Y | \ntok{i_1}{i_n}} \, || \, P^{\theta_2}_{Y | \ntok{i_1}{i_n}} | P_{\ntok{i_1}{i_n}})	\\
	= & ~ \frac{1}{2\sigma^2} \expect{\sum_{i \in I} (\theta_1 - \theta_2)_i^2}	= \frac{1}{2\sigma^2} \sum_{i=1}^d (\theta_1 - \theta_2)_i^2 \expect{\indc{i \in I}}\\
	= & ~ \frac{1}{2\sigma^2} \pth{1 - \pth{1-\frac{1}{d}}^n} \|\theta_1 - \theta_2\|_2^2. 
	\end{align*}
Note that for any $t > 0$, $n \mapsto (1-\frac{t}{n})^n$ is increasing. Therefore $1-(1-\frac{1}{d})^n \leq \frac{n}{d}$ for all $n,d \in \naturals$.
Plugging it in the last display, we obtain the second inequality in \prettyref{lmm:erasure}.
	\end{proof}




\end{document}